\documentclass[12pt,oneside]{article}

\usepackage{hyperref}
\usepackage{amssymb}
\usepackage{amsmath}
\usepackage[letterpaper,left=1in,right=1in, bottom=1in, top=1in]{geometry}
\usepackage{titlesec}
\usepackage{setspace}
\usepackage{textcomp}
\usepackage{amsthm}
\usepackage{eso-pic}
\usepackage{graphicx}
\usepackage{color}
\usepackage{longtable}
\usepackage{transparent}
\usepackage{caption}
\usepackage{hyperref}

\usepackage{tikz}
\usetikzlibrary{shapes,arrows,positioning,quotes}

\newcommand{\mo}[1]{(\text{mod }#1)}

\titleformat{\section}{\bfseries\centering}{\thesection \hspace{0.5cm}}{12pt}{}

\titleformat{\subsection}{\bfseries}{\thesubsection \hspace{0.5cm}}{12pt}{}

\newtheorem{theorem}{Theorem}[section]

\newtheorem{proposition}{Proposition}[section]
\newtheorem{corollary}{Corollary}[section]

\newtheorem{definition}{Definition}[section]

\newtheorem{question}{Question}[section]
\newtheorem{conjecture}{Conjecture}[section]

\begin{document}\setlength{\parindent}{0cm}\thispagestyle{empty}
	\begin{center}
		\Large \textbf{
			{$(a,b)$-Fibonacci-Legendre Cordial Graphs and\\ $k$-Pisano-Legendre Primes}}
	\end{center}
	\vspace{2mm}
	Jason D. Andoyo
	
	University of Southeastern Philippines, Davao City, Republic of the Philippines

	\section*{Abstract}\small
	\hspace{1cm}Let $p$ be an odd prime and let $F_i$ be the $i$th $(a,b)$-Fibonacci number with initial values $F_0=a$ and $F_1=b$. For a simple connected graph $G=(V,E)$, define a bijective function $f:V(G)\to \{0,1,\ldots,|V|-1\}$. If the induced function $f_p^*:E(G)\to \{0,1\}$, defined by $f_p^*(uv)=\frac{1+([F_{f(u)}+F_{f(v)}]/p)}{2}$ whenever $F_{f(u)}+F_{f(v)}\not\equiv 0\mo{p}$ and $f_p^*(uv)=0$ whenever $F_{f(u)}+F_{f(v)}\equiv 0\mo{p}$, satisfies the condition $|e_{f_p^*}(0)-e_{f_p^*}(1)|\leq 1$ where $e_{f_p^*}(i)$ is the number of edges labeled $i$ ($i=0,1$), then $f$ is called $(a,b)$-Fibonacci-Legendre cordial labeling modulo $p$. In this paper, the $(a,b)$-Fibonacci-Legendre cordial labeling of path graphs, star graphs, wheel graphs, and graphs under the operations join, corona, lexicographic product, cartesian product, tensor product, and strong product is explored in relation to $k$-Pisano-Legendre primes relative to $(a,b)$. We also present some properties of $k$-Pisano-Legendre primes relative to $(a,b)$ and numerical observations on its distribution, leading to several conjectures concerning their density and growth behavior.
	
	\textbf{Keywords:} odd prime; $(a,b)$-Fibonacci-Legendre cordial labeling; $k$-Pisano-Legendre primes
	
	\textbf{2020 MSC:} 05C76, 05C78, 11A07, 11A15, 11A41, 11B39
	\normalsize
	
	\setlength{\parindent}{1cm}
	\section{Introduction}
	
	\hspace{1cm}A graph $G = (V, E)$ has a \textit{vertex set} $V = V(G)$ and an \textit{edge set} $E = E(G)$. The elements of $V$ and $E$ are called \textit{vertices} and \textit{edges}, respectively. If $G$ has $n$ vertices and $m$ edges, then $G$ has \textit{order} $n$ and \textit{size} $m$. 
	
	The \textit{Fibonacci sequence}, introduced in the 13th century by Leonardo of Pisa (Fibonacci) in his study of idealized rabbit populations, has since become one of the most celebrated sequences in mathematics due to its elegant recursive structure and deep arithmetic pro\-perties. In this paper, we define the \textit{$(a,b)$-Fibonacci sequence} that satisfies the following recursive relation: 
	$$F_0=a,\quad F_1=b, \quad F_{n}=F_{n-2}+F_{n-1}\text{ for $n\geq 2$},$$ where $F_0=a$ and $F_1=b$ are called initial values. Throughout this paper, we assume that $(a,b)\neq (0,0)$. Hence, $(0,1)$-Fibonacci sequence is the classical Fibonacci sequence and $(2,1)$-Fibonacci sequence is the classical Lucas sequence.
	
	The \textit{Legendre symbol}, introduced by Adrien-Marie Legendre in the late 18th century, serves as a compact notation for determining whether an integer is a quadratic residue of an odd prime $p$, taking the value  
	$$
	(a/p) = 
	\begin{cases} 
		1 & \text{if $a$ is a quadratic residue of $p$},\\[2mm]
		-1 & \text{if $a$ is a quadratic nonresidue of $p$},\\[1mm]
		0 & \text{if $a\equiv 0\mo{p}$}.
	\end{cases}
	$$  
	One important property of the Legendre symbol is that if $a$ and $b$ are integers and $a\equiv b\mo{p}$, then $(a/p)=(b/p)$. 
	
	These two classical mathematical concepts play important roles across number theory: Fibonacci numbers appear in combinatorics, algebra, and modular recurrences, while the Legendre symbol underlies quadratic reciprocity, primality tests, and the structure of congruences. Together, their interaction gives rise to intriguing problems involving periodicity and arithmetic behavior modulo primes.
	
	Another mathematical concept is \textit{graph labeling}, which involves assigning labels, usually numbers, to the vertices or edges of a graph according to specific conditions and has applications in combinatorics, network theory, and the study of structural patterns \cite{Gallian}. One important type is \textit{cordial labeling}, introduced by Ibrahim Cahit \cite{Cahit}, which aims to balance the distribution of labels across vertices and edges. This concept led to the creation of Legendre cordial labeling \cite{Andoyo3}, which is the first graph labeling that explicitly uses the term \textit{modulo} in its labeling function, namely modulo an odd prime $p$. This was later extended by Euler cordial labeling \cite{Andoyo1}, Legendre product cordial labeling \cite{Prudencio}, and logarithmic cordial labeling \cite{Andoyo2}, where different concepts of Number Theory have been used such as Euler's Theorem, multiplicative property of the Legendre symbol, and discrete logarithm (indices). Moreover, cordial labeling also motivated the notions of Fibonacci cordial labeling \cite{Sulayman} and Lucas cordial labeling \cite{Salise}, where the concepts of the Fibonacci sequence and the Lucas sequence are used in terms of parity conditions. All the above-mentioned variants of cordial labeling eventually led to the discovery of $(a,b)$-Fibonacci-Legendre cordial labeling, which uses the $(a,b)$-Fibonacci sequence and the Legendre symbol, and whose labeling function operates modulo an odd prime $p$.
	
	Prime numbers play a fundamental role in number theory and its applications. The importance of primes is highlighted by the Fundamental Theorem of Arithmetic, which states that every integer greater than $1$ can be expressed uniquely (up to ordering) as a product of prime numbers. Moreover, Euclid proved that there are infinitely many prime numbers, establishing their inexhaustible nature. Over time, several notable families of primes have been studied due to their special algebraic forms and properties, including Mersenne primes \cite{Toledo} of the form $2^{p}-1$, where $p$ is prime, Fermat primes \cite{Brandli} of the form $2^{2^{n}}+1$, and Sophie Germain primes \cite{Li}, where a prime $p$ satisfies that $2p+1$ is also prime. These special classes illustrate how primes naturally arise from arithmetic structures. Motivated by these concepts, a new family of primes, called \textit{$k$-Pisano-Legendre primes relative to $(a,b)$}, is introduced in this study. This family is defined through the interaction of the $(a,b)$-Fibonacci sequence and its corresponding $(a,b)$-Pisano period of the given $k$-Pisano-Legendre prime (see Definition~\ref{PL}).

	The present work is primarily constructive in nature. In particular, we investigate the $(a,b)$-Fibonacci-Legendre cordial labeling on several standard families of graphs and under common graph operations. In addition, we examine several properties of $k$-Pisano-Legendre primes relative to $(a,b)$ and present some numerical observations. These findings lead to the formulation of conjectures that highlight emerging patterns and open directions for further research.
	\section{Basic Concepts}
	\begin{definition}\normalfont
		A \textit{path graph} $P_n$ of order $n$ and size $n-1$ is a graph with vertex set $V(P_n)=\{v_1,v_2,\ldots,v_n\}$ and edge set $E(P_n)=\{v_1v_2,v_2v_3,\ldots,v_{n-1}v_n\}$. In addition, a \textit{cycle graph} $C_n$ of order $n$ and size $n$ is obtained from path graph $P_n$ with an additional edge $v_1v_n$.
	\end{definition}
	
	\begin{definition}\normalfont
		A \textit{star graph} $S_n$ of order $n$ and size $n-1$ is a graph with vertex set $V(S_n)=\{v_1,v_2,\ldots,v_{n-1}\}\cup \{x_0\}$ and edge set $E(S_n)=\{v_ix_0:i=1,2,\ldots,n-1\}$, where $x_0$ is called a \textit{central vertex} of $S_n$.
	\end{definition}
	
	\begin{definition}\normalfont
		A \textit{wheel graph} $W_n$ of order $n$ is obtained from a cycle graph $C_{n-1}$ such that every vertex of $C_{n-1}$ is adjacent to a new vertex $x_0$
	\end{definition}
	
	\begin{definition}\normalfont
		Let $G^1,G^2,\ldots,G^n$ be graphs. Then the \textit{union graph} of $G^1,G^2,\ldots,G^n$, denoted by $\bigcup_{i=1}^nG^i=G^1\cup G^2\cup \cdots\cup G^n,$ 
		is a graph with vertex set
		$$V(G^1)\cup V(G^2)\cup\cdots\cup V(G^n)$$
		and edge set
		$$E(G^1)\cup E(G^2)\cup\cdots\cup E(G^n).$$
	\end{definition}
	
	\begin{definition}\normalfont
		Let $G_1$ and $G_2$ be graphs. Note that the vertex set of the graphs defined in items (iii) to (vi) is $V(G_1)\times V(G_2)$.
		\begin{enumerate}
			\item [i.] A \textit{join graph} $G_1+G_2$ is a graph with vertex set 
			\begin{equation*}
				V(G_1+G_2)=V(G_1)\cup V(G_2)
			\end{equation*}
			and edge set
			\begin{equation*}
				E(G_1+G_2)=E(G_1)\cup E(G_2)\cup \{uv:u\in V(G_1)\text{ and }v\in V(G_2)\}.
			\end{equation*}
			\item[ii.] If $G_1$ has order $n$, then the \textit{corona graph} $G_1\circ G_2$ is obtained by taking one copy of $G_1$ and $n$ copies of $G_2$, and then joining the $i$th vertex of $G_1$ to every vertex of the $i$th copy of $G_2$. 
			\item[iii.] A \textit{lexicographic product graph} $G_1\otimes G_2$ is a graph with edge set
			\begin{align*}
				E(G_1\otimes G_2)&=\{(v_1,u_1)(v_2,u_2): v_1v_2\in E(G_1),\\
				&\text{\hspace{0.5cm} or }v_1=v_2 \text{ and }u_1u_2\in E(G_2)\}.
			\end{align*}
			\item[iv.] A \textit{cartesian product graph} $G_1\square G_2$ is a graph with edge set
			\begin{align*}
				E(G_1\square G_2)&=\lbrace(v_1,u_1)(v_2,u_2): v_1=v_2\text{ and }u_1u_2\in E(G_2),\\
				&\text{\hspace{0.5cm}}\text{ or } u_1=u_2\text{ and }v_1v_2\in E(G_1)\rbrace.
			\end{align*}
			\item[v.] A \textit{tensor product graph} $G_1\times G_2$ is a graph with edge set
			\begin{equation*}
				E(G_1\times G_2)=\{(v_1,u_1)(v_2,u_2): v_1v_2\in E(G_1)\text{ and }u_1u_2 \in E(G_2)\}.
			\end{equation*}
			\item[vi.] A \textit{strong product graph} $G_1\boxtimes G_2$ is a graph with edge set
			\begin{equation*}
				E(G_1\boxtimes G_2)=E(G_1\square G_2)\cup E(G_1\times G_2).
			\end{equation*}
		\end{enumerate}
	\end{definition}
	
	Consider the path graph $P_3$ and cycle graph $C_3$ shown in Figures~\ref{P3} and \ref{C3}, respectively. Thus, the union graph $C_3\cup P_3$, join graph $C_3+P_3$, corona graph $C_3\circ P_3$, lexicographic product graph $C_3\otimes P_3$, cartesian product graph $C_3\square P_3$, tensor product graph $C_3\times P_3$, and strong product graph $C_3\boxtimes P_3$ are illustrated in Figures~\ref{union}, \ref{join}, \ref{corona}, \ref{lex}, \ref{cart}, \ref{ten}, and \ref{str}, respectively.
	
	\vspace{0.5cm}
	
	\begin{figure}[hbt!]\centering
		\begin{tikzpicture}[scale=1.2]
			\node[circle,draw,scale=0.4,label=90:$v_1$] (v1) at (0,0) {};
			\node[circle,draw,scale=0.4,label=90:$v_2$] (v2) at (1,0) {};
			\node[circle,draw,scale=0.4,label=90:$v_3$] (v3) at (2,0) {};
			\draw[-] (v1)--(v2)--(v3);
			\node[rectangle] () at (-1,0) {$P_3:$};
		\end{tikzpicture}
		\caption{Path graph $P_3$}\label{P3}
	\end{figure}
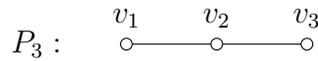

	\begin{figure}[hbt!]\centering
		\begin{tikzpicture}[scale=1.2]
			\node[circle,draw,scale=0.4,label=180:$u_1$] (v1) at (0,0) {};
			\node[circle,draw,scale=0.4,label=90:$u_2$] (v2) at (1,1) {};
			\node[circle,draw,scale=0.4,label=0:$u_3$] (v3) at (2,0) {};
			\draw[-] (v1)--(v2)--(v3)--(v1);
			\node[rectangle] () at (-1,0.5) {$C_3:$};
		\end{tikzpicture}
		\caption{Cycle graph $C_3$}\label{C3}
	\end{figure}
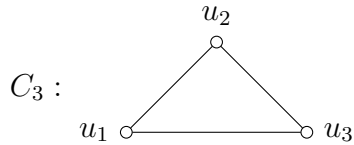
	
	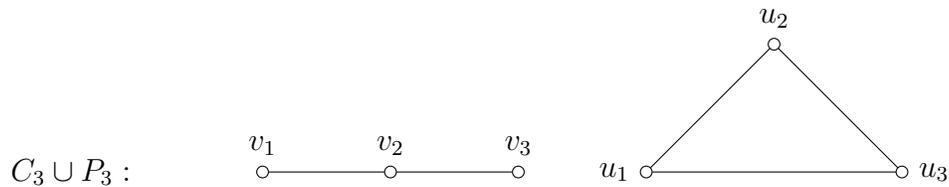
\begin{figure}[!hbt]\centering
		\begin{tikzpicture}[scale=1.7]
			\node[circle,draw,scale=0.4,label=90:$v_1$] (v1) at (0,0) {};
			\node[circle,draw,scale=0.4,label=90:$v_2$] (v2) at (1,0) {};
			\node[circle,draw,scale=0.4,label=90:$v_3$] (v3) at (2,0) {};
			\draw[-] (v1)--(v2)--(v3);
			
			\node[circle,draw,scale=0.4,label=180:$u_1$] (u1) at (3,0) {};
			\node[circle,draw,scale=0.4,label=90:$u_2$] (u2) at (4,1) {};
			\node[circle,draw,scale=0.4,label=0:$u_3$] (u3) at (5,0) {};
			\draw[-] (u1)--(u2)--(u3)--(u1);

			\node[rectangle] () at (-1.5,0) {$C_3\cup P_3:$};
		\end{tikzpicture}
		\caption{Union graph $C_3\cup P_3$}\label{union}
	\end{figure}

	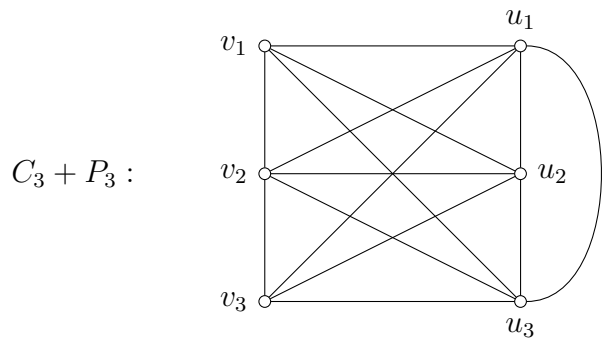
\begin{figure}[!hbt]\centering
		\begin{tikzpicture}[scale=1.7]
			\node[circle,draw,scale=0.4,label=180:$v_1$] (v1) at (0,2) {};
			\node[circle,draw,scale=0.4,label=180:$v_2$] (v2) at (0,1) {};
			\node[circle,draw,scale=0.4,label=180:$v_3$] (v3) at (0,0) {};
			\draw[-] (v1)--(v2)--(v3);
			
			\node[circle,draw,scale=0.4,label=90:$u_1$] (u1) at (2,2) {};
			\node[circle,draw,scale=0.4,label=0:$u_2$] (u2) at (2,1) {};
			\node[circle,draw,scale=0.4,label=270:$u_3$] (u3) at (2,0) {};
			\draw[-] (u1)--(u2)--(u3);
			\draw[-] (u3)edge[in=0,out=0](u1);
			
			\draw[-] (v1)--(u1);
			\draw[-] (v1)--(u2);
			\draw[-] (v1)--(u3);

			\draw[-] (v2)--(u1);
			\draw[-] (v2)--(u2);
			\draw[-] (v2)--(u3);

			\draw[-] (v3)--(u1);
			\draw[-] (v3)--(u2);
			\draw[-] (v3)--(u3);
			
			\node[rectangle] () at (-1.5,1) {$C_3+ P_3:$};
		\end{tikzpicture}
		\caption{Join graph $C_3+P_3$}\label{join}
	\end{figure}

	\begin{figure}[!hbt]\centering
		\begin{tikzpicture}[scale=1.1]
			
			\node[circle,draw,scale=0.4,label=180:$u_1$] (u1) at (3,0) {};
			\node[circle,draw,scale=0.4,label=270:$u_2$] (u2) at (4,1) {};
			\node[circle,draw,scale=0.4,label=0:$u_3$] (u3) at (5,0) {};
			\draw[-] (u1)--(u2)--(u3)--(u1);
			
			\node[circle,draw,scale=0.4,label=270:$v_1^1$] (v11) at (1,-1) {};
			\node[circle,draw,scale=0.4,label=270:$v_2^1$] (v21) at (2,-1) {};
			\node[circle,draw,scale=0.4,label=270:$v_3^1$] (v31) at (3,-1) {};
			\draw[-] (v11)--(v21)--(v31);
			\draw[-] (u1)--(v11);
			\draw[-] (u1)--(v21);
			\draw[-] (u1)--(v31);

			\node[circle,draw,scale=0.4,label=90:$v_1^2$] (v12) at (3,2) {};
			\node[circle,draw,scale=0.4,label=90:$v_2^2$] (v22) at (4,2) {};
			\node[circle,draw,scale=0.4,label=90:$v_3^2$] (v32) at (5,2) {};
			\draw[-] (v12)--(v22)--(v32);
			\draw[-] (u2)--(v12);
			\draw[-] (u2)--(v22);
			\draw[-] (u2)--(v32);

			\node[circle,draw,scale=0.4,label=270:$v_1^3$] (v13) at (5,-1) {};
			\node[circle,draw,scale=0.4,label=270:$v_2^3$] (v23) at (6,-1) {};
			\node[circle,draw,scale=0.4,label=270:$v_3^3$] (v33) at (7,-1) {};
			\draw[-] (v13)--(v23)--(v33);
			\draw[-] (u3)--(v13);
			\draw[-] (u3)--(v23);
			\draw[-] (u3)--(v33);
			
			\node[rectangle] () at (0,1) {$C_3\circ P_3:$};
		\end{tikzpicture}
		\caption{Corona graph $C_3\circ P_3$}\label{corona}
	\end{figure}
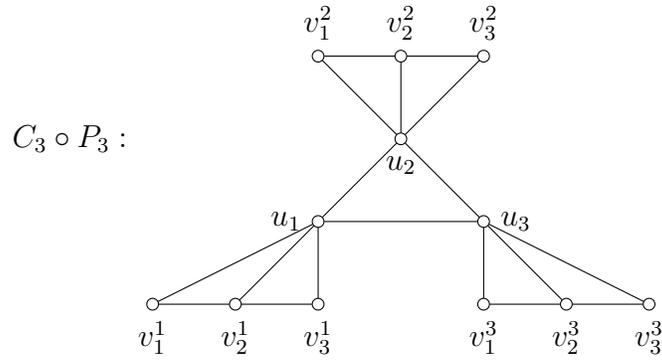
	
	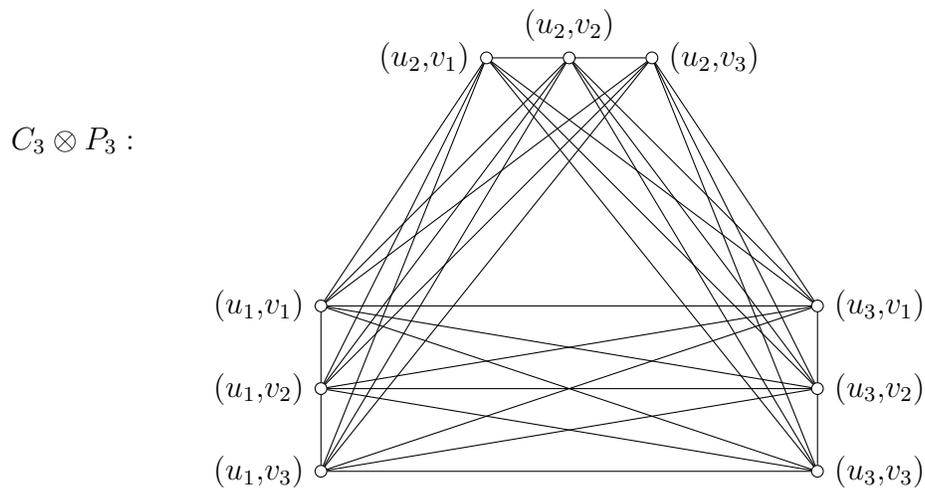
\begin{figure}[!hbt]\centering
		\begin{tikzpicture}[scale=1.1]

			\node[circle,draw,scale=0.4,label=180:$(u_1\text{$,$}v_1)$] (v11) at (1,-1) {};
			\node[circle,draw,scale=0.4,label=180:$(u_1\text{$,$}v_2)$] (v21) at (1,-2) {};
			\node[circle,draw,scale=0.4,label=180:$(u_1\text{$,$}v_3)$] (v31) at (1,-3) {};
			\draw[-] (v11)--(v21)--(v31);

			\node[circle,draw,scale=0.4,label=180:$(u_2\text{$,$}v_1)$] (v12) at (3,2) {};
			\node[circle,draw,scale=0.4,label=90:$(u_2\text{$,$}v_2)$] (v22) at (4,2) {};
			\node[circle,draw,scale=0.4,label=0:$(u_2\text{$,$}v_3)$] (v32) at (5,2) {};
			\draw[-] (v12)--(v22)--(v32);

			\node[circle,draw,scale=0.4,label=0:$(u_3\text{$,$}v_1)$] (v13) at (7,-1) {};
			\node[circle,draw,scale=0.4,label=0:$(u_3\text{$,$}v_2)$] (v23) at (7,-2) {};
			\node[circle,draw,scale=0.4,label=0:$(u_3\text{$,$}v_3)$] (v33) at (7,-3) {};
			\draw[-] (v13)--(v23)--(v33);
			
			\node[rectangle] () at (-2,1) {$C_3\otimes P_3:$};

			\draw[-] (v11)--(v12);
			\draw[-] (v11)--(v22);
			\draw[-] (v11)--(v32);

			\draw[-] (v21)--(v12);
			\draw[-] (v21)--(v22);
			\draw[-] (v21)--(v32);

			\draw[-] (v31)--(v12);
			\draw[-] (v31)--(v22);
			\draw[-] (v31)--(v32);

			\draw[-] (v11)--(v13);
			\draw[-] (v11)--(v23);
			\draw[-] (v11)--(v33);

			\draw[-] (v21)--(v13);
			\draw[-] (v21)--(v23);
			\draw[-] (v21)--(v33);

			\draw[-] (v31)--(v13);
			\draw[-] (v31)--(v23);
			\draw[-] (v31)--(v33);

			\draw[-] (v13)--(v12);
			\draw[-] (v13)--(v22);
			\draw[-] (v13)--(v32);

			\draw[-] (v23)--(v12);
			\draw[-] (v23)--(v22);
			\draw[-] (v23)--(v32);

			\draw[-] (v33)--(v12);
			\draw[-] (v33)--(v22);
			\draw[-] (v33)--(v32);
		\end{tikzpicture}
		\caption{Lexicographic product graph $C_3\otimes P_3$}\label{lex}
	\end{figure}

	\begin{figure}[!hbt]\centering
		\begin{tikzpicture}[scale=1.1]

			\node[circle,draw,scale=0.4,label=180:$(u_1\text{$,$}v_1)$] (v11) at (1,-1) {};
			\node[circle,draw,scale=0.4,label=180:$(u_1\text{$,$}v_2)$] (v21) at (1,-2) {};
			\node[circle,draw,scale=0.4,label=180:$(u_1\text{$,$}v_3)$] (v31) at (1,-3) {};
			\draw[-] (v11)--(v21)--(v31);

			\node[circle,draw,scale=0.4,label=180:$(u_2\text{$,$}v_1)$] (v12) at (3,2) {};
			\node[circle,draw,scale=0.4,label=90:$(u_2\text{$,$}v_2)$] (v22) at (4,2) {};
			\node[circle,draw,scale=0.4,label=0:$(u_2\text{$,$}v_3)$] (v32) at (5,2) {};
			\draw[-] (v12)--(v22)--(v32);

			\node[circle,draw,scale=0.4,label=0:$(u_3\text{$,$}v_1)$] (v13) at (7,-3) {};
			\node[circle,draw,scale=0.4,label=0:$(u_3\text{$,$}v_2)$] (v23) at (7,-2) {};
			\node[circle,draw,scale=0.4,label=0:$(u_3\text{$,$}v_3)$] (v33) at (7,-1) {};
			\draw[-] (v13)--(v23)--(v33);
			
			\node[rectangle] () at (-2,1) {$C_3\square P_3:$};

			\draw[-] (v11)--(v12)--(v13)--(v11);
			\draw[-] (v21)--(v22)--(v23)--(v21);
			\draw[-] (v31)--(v32)--(v33)--(v31);
		\end{tikzpicture}
		\caption{Cartesian product graph $C_3\square P_3$}\label{cart}
	\end{figure}
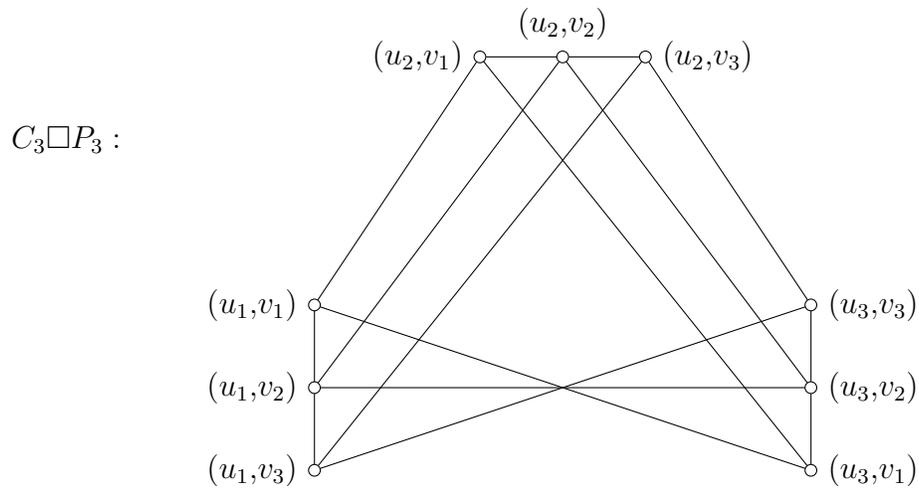

	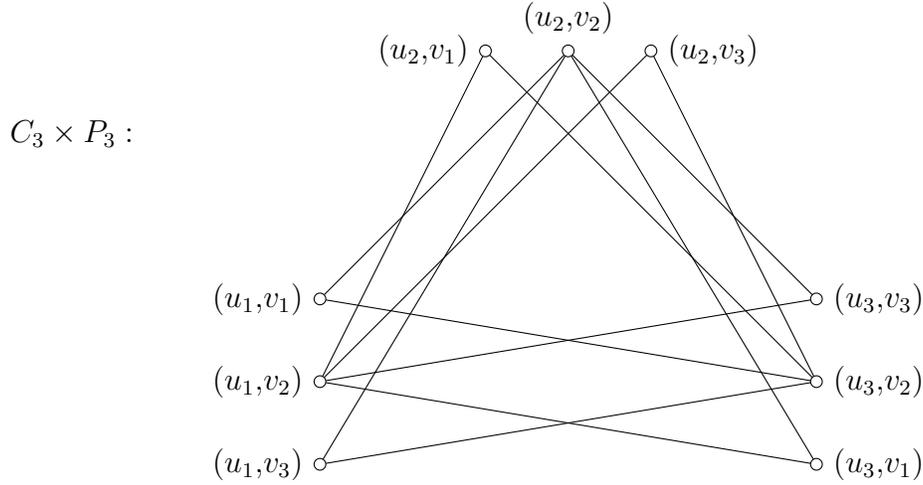
\begin{figure}[!hbt]\centering
		\begin{tikzpicture}[scale=1.1]

			\node[circle,draw,scale=0.4,label=180:$(u_1\text{$,$}v_1)$] (v11) at (1,-1) {};
			\node[circle,draw,scale=0.4,label=180:$(u_1\text{$,$}v_2)$] (v21) at (1,-2) {};
			\node[circle,draw,scale=0.4,label=180:$(u_1\text{$,$}v_3)$] (v31) at (1,-3) {};

			\node[circle,draw,scale=0.4,label=180:$(u_2\text{$,$}v_1)$] (v12) at (3,2) {};
			\node[circle,draw,scale=0.4,label=90:$(u_2\text{$,$}v_2)$] (v22) at (4,2) {};
			\node[circle,draw,scale=0.4,label=0:$(u_2\text{$,$}v_3)$] (v32) at (5,2) {};

			\node[circle,draw,scale=0.4,label=0:$(u_3\text{$,$}v_1)$] (v13) at (7,-3) {};
			\node[circle,draw,scale=0.4,label=0:$(u_3\text{$,$}v_2)$] (v23) at (7,-2) {};
			\node[circle,draw,scale=0.4,label=0:$(u_3\text{$,$}v_3)$] (v33) at (7,-1) {};

			\node[rectangle] () at (-2,1) {$C_3\times P_3:$};
			
			\draw[-] (v11)--(v22)--(v31);
			\draw[-] (v12)--(v21)--(v32);
			\draw[-] (v11)--(v23)--(v31);
			\draw[-] (v13)--(v21)--(v33);
			\draw[-] (v12)--(v23)--(v32);
			\draw[-] (v13)--(v22)--(v33);
		\end{tikzpicture}
		\caption{Tensor product graph $C_3\times P_3$}\label{ten}
	\end{figure}

	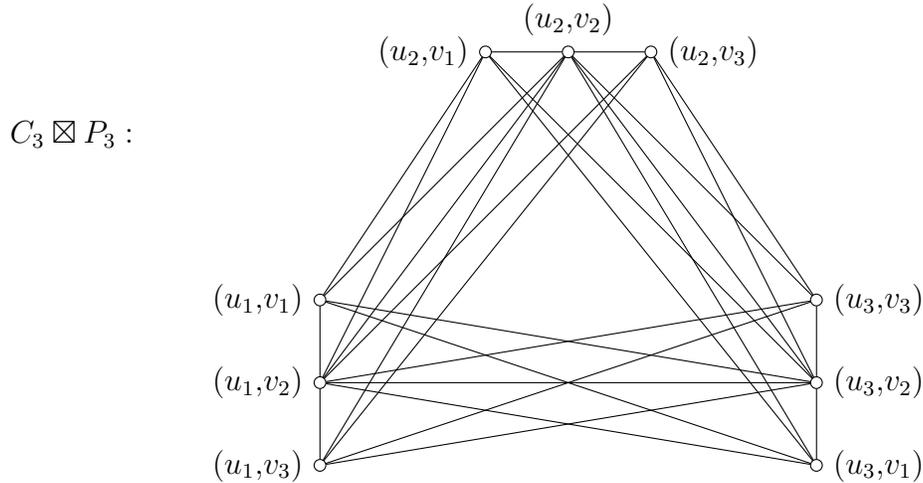
\begin{figure}[!hbt]\centering
		\begin{tikzpicture}[scale=1.1]

			\node[circle,draw,scale=0.4,label=180:$(u_1\text{$,$}v_1)$] (v11) at (1,-1) {};
			\node[circle,draw,scale=0.4,label=180:$(u_1\text{$,$}v_2)$] (v21) at (1,-2) {};
			\node[circle,draw,scale=0.4,label=180:$(u_1\text{$,$}v_3)$] (v31) at (1,-3) {};
			\draw[-] (v11)--(v21)--(v31);

			\node[circle,draw,scale=0.4,label=180:$(u_2\text{$,$}v_1)$] (v12) at (3,2) {};
			\node[circle,draw,scale=0.4,label=90:$(u_2\text{$,$}v_2)$] (v22) at (4,2) {};
			\node[circle,draw,scale=0.4,label=0:$(u_2\text{$,$}v_3)$] (v32) at (5,2) {};
			\draw[-] (v12)--(v22)--(v32);

			\node[circle,draw,scale=0.4,label=0:$(u_3\text{$,$}v_1)$] (v13) at (7,-3) {};
			\node[circle,draw,scale=0.4,label=0:$(u_3\text{$,$}v_2)$] (v23) at (7,-2) {};
			\node[circle,draw,scale=0.4,label=0:$(u_3\text{$,$}v_3)$] (v33) at (7,-1) {};
			\draw[-] (v13)--(v23)--(v33);
			
			\node[rectangle] () at (-2,1) {$C_3\boxtimes P_3:$};
			
			\draw[-] (v11)--(v22)--(v31);
			\draw[-] (v12)--(v21)--(v32);
			\draw[-] (v11)--(v23)--(v31);
			\draw[-] (v13)--(v21)--(v33);
			\draw[-] (v12)--(v23)--(v32);
			\draw[-] (v13)--(v22)--(v33);
			
			\draw[-] (v11)--(v12)--(v13)--(v11);
			\draw[-] (v21)--(v22)--(v23)--(v21);
			\draw[-] (v31)--(v32)--(v33)--(v31);
		\end{tikzpicture}
		\caption{Strong product graph $C_3\boxtimes P_3$}\label{str}
	\end{figure}

	\begin{definition}\normalfont\cite{Guyer}
		The \textit{$(a,b)$-Pisano period} of a positive integer $m$, denoted by $\pi_m(a,b)$ (where $F_0=a$ and $F_1=b$ are initial values), is the smallest positive integer $k$ such that $F_{n+k}\equiv F_n\mo{m}$ for all $n$.
	\end{definition}
	
	\begin{definition}\normalfont\label{PL}
		Let $p$ be an odd prime and let $S=\{0,1,\ldots,\pi_p(a,b)-1\}$ where $a$ and $b$ are the initial values. Suppose that $\Lambda_j^p(a,b)=\{\xi\in S:(F_{\xi}/p)=j\}$ for $j=-1,0,1$. If $k=|\Lambda_1^p(a,b)|-|\Lambda_{-1}^p(a,b)|-|\Lambda_0^p(a,b)|$, then $p$ is called \textit{$k$-Pisano-Legendre prime relative to $(a,b)$} (abbreviated by \textit{$k$-PL prime relative to $(a,b)$}). The \textit{Pisano set of $(a,b)$}, denoted by $\chi(a,b)$ is the set of all integers $k$ such that there exists an odd prime $p$ for which $p$ is a $k$-PL prime relative to $(a,b)$. 
	\end{definition}
	
	Consider the odd prime $p=3$. It is easy to show that the $(0,1)$-Pisano period of $3$ is $\pi_3(0,1)=8$. Now, we evaluate the Legendre symbol $(F_i/3)$ for $i=0,1,\ldots,7$. Note that $(0/3)=0$, $(1/3)=1$, and $(2/3)=-1$. Consider Table~\ref{tab1}.
	\begin{table}[h]\centering
		\begin{tabular}{|c|c|c|c|c|c|c|c|c|}
			\hline
			$i$&$0$&$1$&$2$&$3$&$4$&$5$&$6$&$7$\\
			\hline
			$F_i$&$0$&$1$&$1$&$2$&$3$&$5$&$8$&$13$\\
			\hline
			$(F_i/3)$&$0$&$1$&$1$&$-1$&$0$&$-1$&$-1$&$1$\\
			\hline
		\end{tabular}
		\caption{Legendre symbol $(F_i/3)$ for $i=0,1,\ldots,7$}\label{tab1}
	\end{table}\\
	Hence, $3,5,6\in \Lambda_{-1}^3(0,1)$, $0,4\in \Lambda_{0}^3(0,1)$, and $1,2,7\in \Lambda_{1}^3(0,1)$. So,
	$$|\Lambda_{1}^3(0,1)|-|\Lambda_{-1}^3(0,1)|-|\Lambda_{0}^3(0,1)|=3-3-2=-2.$$
	Therefore, $3$ is a $(-2)$-PL prime relative to $(0,1)$. This means that $-2\in \chi(0,1)$.
	
	\begin{definition}\normalfont
		Let $G$ be a simple connected graph of order $n$ and let $F_0=a$ and $F_1=b$ be the initial values. A bijective function $f:V(G)\to\{0,1,\ldots,n-1\}$ is called \textit{$(a,b)$-Fibonacci-Legendre cordial labeling modulo $p$} (abbreviated \textit{$(a,b)$-FLC labeling modulo $p$}), where $p$ is an odd prime, if the induced function $f_p^*:E(G)\to \{0,1\}$, defined by
		$$f_p^*(uv)=\begin{cases}
			\frac{1+([F_{f(u)}+F_{f(v)}]/p)}{2}&\text{ if }F_{f(u)}+F_{f(v)}\not\equiv0\mo{p}\\
			0&\text{ otherwise},
		\end{cases}$$
		satisfies the condition $|e_{f_p^*}(0)-e_{f_p^*}(1)|\leq 1$ where $e_{f_p^*}(i)$ is the number of edges labeled $i$ ($i=0,1$). A graph that admits this labeling is called \textit{$(a,b)$-FLC graph modulo $p$}.
	\end{definition}
	
	As an example, consider the cycle graph $C_3$ shown in Figure~\ref{C3}. In this case, consider the odd prime $p=5$, and the integers $a=7$ and $b=4$. So, we consider the first three terms of $(7,4)$-Fibonacci sequence, namely, $F_0=7$, $F_1=4$, and $F_2=11$. Let $f:V(C_3)\to \{0,1,2\}$ be a function defined by
	$$f(u_1)=0,\quad f(u_2)=1,\quad f(u_3)=2.$$
	Hence, $f$ is a bijective function. Note that $(1/5)=1$, $(3/5)=-1$, and $(0/5)=0$. For edge labels, we have
	\begin{align*}
		u_1u_2:\text{ }([F_{f(u_1)}+F_{f(u_2)}]/5)=([F_0+F_1]/5)=(11/5)=(1/5)&=1\\
		u_1u_3:\text{ }([F_{f(u_1)}+F_{f(u_3)}]/5)=([F_0+F_2]/5)=(18/5)=(3/5)&=-1\\
		u_2u_3:\text{ }([F_{f(u_2)}+F_{f(u_3)}]/5)=([F_1+F_2]/5)=(15/5)=(0/5)&=0.
	\end{align*}
	Therefore,
	$$f_5^*(u_1u_2)=1,\quad f_5^*(u_1u_3)=0,\quad f_5^*(u_2u_3)=0.$$
	Consequently, $e_{f_p^*}(0)=2$ and $e_{f_p^*}(1)=1$ and it follows that $|e_{f_p^*}(0)-e_{f_p^*}(1)|=1$. This implies that $C_3$ is a $(7,4)$-FLC graph modulo $5$. Figure~\ref{C3FLC} demonstrates the $(7,4)$-FLC labeling modulo $5$ of $C_3$.
	
	\begin{figure}[h]\centering
		\begin{tikzpicture}
			\node[circle,draw,scale=0.4,label=180:$u_1\mapsto 0$] (v1) at (0,0) {};
			\node[circle,draw,scale=0.4,label=90:$u_2\mapsto 1$] (v2) at (1,1) {};
			\node[circle,draw,scale=0.4,label=0:$u_3\mapsto 2$] (v3) at (2,0) {};
			\draw[-,color=blue] (v1)edge["$11$"](v2);
			\draw[-,color=red] (v3)edge["$18$"](v1);
			\draw[-,color=red] (v2)edge["$15$"](v3);
		\end{tikzpicture}
		\caption{$(7,4)$-FLC labeling modulo $5$ of $C_3$}\label{C3FLC}
	\end{figure}
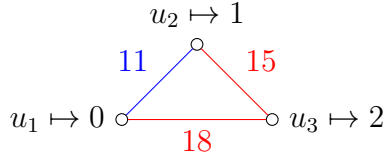
	
	\begin{theorem}\label{multiplicative}\normalfont\cite{Burton,Rosen}
		Suppose that $a$ and $b$ are integers. If $p$ is an odd prime, then $(ab/p)=(a/p)(b/p)$.
	\end{theorem}
	
	\begin{theorem}\label{2/p}\normalfont\cite{Burton,Rosen}
		Let $p$ be an odd prime. Then
		$$(2/p)=\begin{cases}
			-1&\text{ if }p\equiv\pm 3\mo{8}\\
			1&\text{ if }p\equiv\pm 1\mo{8}.
		\end{cases}$$
	\end{theorem}
	
	\section{$(a,b)$-Fibonacci-Legendre Cordial Labeling}\label{sec3}
	
	\hspace{1cm}In this section, we denote $p$ as odd prime. All terminologies of Definition~\ref{PL} will be used in every proof. In addition, suppose that $F_0=a$ and $F_1=b$ are the initial values. Furthermore, it should be noted that $F_{m+n\pi_p(a,b)}\equiv F_{m}\mo{p}$ for all $m,n\geq 1$.
	
	\begin{theorem}\normalfont\label{thm1}
		Let $0\in\chi(a,b)$. The path graph $P_{q\pi_p(a,b)+1}$ is an $(a,b)$-FLC graph modulo $p$, where $p$ is a $0$-PL prime relative to $(a,b)$, for all $q\geq 1$.
	\end{theorem}
	
	\begin{proof}
		Let $V(P_{q\pi_p(a,b)+1})=\{v_0,v_1,\ldots,v_{q\pi_p(a,b)}\}$ (see Figure~\ref{Pqpi}). In addition, suppose that $f:V(P_{q\pi_p(a,b)+1})\to \{0,1,\ldots,q\pi_p(a,b)\}$ is a function defined by
		$$f(v_i)=i\text{ for }i=0,1,\ldots,q\pi_p(a,b).$$
		Hence, $f$ is a bijective function. 
		
		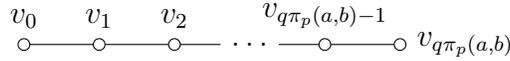
\begin{figure}[h]\centering
			\begin{tikzpicture}
				\node[circle,draw,scale=0.4,label=90:$v_0$] (v1) at (0,0) {};
				\node[circle,draw,scale=0.4,label=90:$v_1$] (v2) at (1,0) {};
				\node[circle,draw,scale=0.4,label=90:$v_2$] (v3) at (2,0) {};
				\node[rectangle] (cd) at (3,0) {$\ldots$};
				\node[circle,draw,scale=0.4,label=90:$v_{q\pi_p(a,b)-1}$] (vqpi-1) at (4,0) {};
				\node[circle,draw,scale=0.4,label=0:$v_{q\pi_p(a,b)}$] (vqpi) at (5,0) {};
				\draw[-] (v1)--(v2)--(v3)--(cd)--(vqpi-1)--(vqpi);
			\end{tikzpicture}
			\caption{Path graph $P_{q\pi_p(a,b)+1}$}\label{Pqpi}
		\end{figure}

		Suppose that 
		$$\mathbb{F}(i,j)=F_{f(v_{i+(j-1)\pi_p(a,b)})}+F_{f(v_{i+1+(j-1)\pi_p(a,b)})}.$$
		Now, for each $j=1,2,\ldots,q$, we have
		\begin{align*}
			\mathbb{F}(i,j)&= F_{i+(j-1)\pi_p(a,b)}+F_{i+1+(j-1)\pi_p(a,b)}\\
			&\equiv F_{i}+F_{i+1}\mo{p}\\
			&\equiv F_{i+2}\mo{p}\\
			&\equiv \begin{cases}
				F_{i+2}\mo{p}&\text{ for }i=0,1,\ldots,\pi_p(a,b)-3\\
				F_0\mo&\text{ for }i=\pi_p(a,b)-2\\
				F_1\mo{p}&\text{ for }i=\pi_p(a,b)-1.
			\end{cases}
		\end{align*}
		Taking the Legendre symbol $(F_i/p)$ for $i=0,1,\ldots,\pi_p(a,b)-1$, we have
		$$e_{f_p^*}(0)=q[|\Lambda_{-1}^p(a,b)|+|\Lambda_{0}^p(a,b)|]\text{ and }e_{f_p^*}(1)=q|\Lambda_{1}^p(a,b)|.$$
		Since $p$ is a $0$-PL prime relative to $(a,b)$, $|e_{f_p^*}(0)-e_{f_p^*}(1)|=0$. Therefore, $P_{q\pi_p(a,b)+1}$ is an $(a,b)$-FLC graph modulo $p$.
	\end{proof}
	
	\begin{theorem}\normalfont\label{thm2}
		Let $0\in\chi(0,b)$. The star graph $S_{q\pi_p(0,b)+1}$ is a $(0,b)$-FLC graph modulo $p$, where $p$ is a $0$-PL prime relative to $(0,b)$, for all $q\geq 1$.
	\end{theorem}
	
	\begin{proof}
		Assume that $V(S_{q\pi_p(0,b)+1})=\{v_0,v_1,\ldots,v_{q\pi_p(0,b)}\}$ where $v_0$ is the central vertex (see Figure~\ref{Sqpi}). Let $f:V(S_{q\pi_p(0,b)+1})\to \{0,1,\ldots,q\pi_p(0,b)\}$ be a function defined by
		$$f(v_i)=i\text{ for }i=0,1,\ldots,q\pi_p(0,b).$$
		Clearly, $f$ is a bijective function.
		
		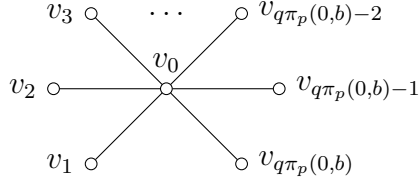
\begin{figure}[h]\centering
			\begin{tikzpicture}
				\node[circle,draw,scale=0.4,label=90:$v_0$] (v0) at (0,0) {};
				\node[circle,draw,scale=0.4,label=180:$v_1$] (v1) at (-1,-1) {};
				\node[circle,draw,scale=0.4,label=180:$v_2$] (v2) at (-1.5,0) {};
				\node[circle,draw,scale=0.4,label=180:$v_3$] (v3) at (-1,1) {};
				\node[rectangle] (cd) at (0,1) {$\ldots$};
				\node[circle,draw,scale=0.4,label=0:$v_{q\pi_p(0,b)-2}$] (vqpi-2) at (1,1) {};
				\node[circle,draw,scale=0.4,label=0:$v_{q\pi_p(0,b)-1}$] (vqpi-1) at (1.5,0) {};
				\node[circle,draw,scale=0.4,label=0:$v_{q\pi_p(0,b)}$] (vqpi) at (1,-1) {};
				
				\draw[-] (v0)--(v1);
				\draw[-] (v0)--(v2);
				\draw[-] (v0)--(v3);
				\draw[-] (v0)--(vqpi-2);
				\draw[-] (v0)--(vqpi-1);
				\draw[-] (v0)--(vqpi);
			\end{tikzpicture}
			\caption{Star graph $S_{q\pi_p(a,b)+1}$}\label{Sqpi}
		\end{figure}
		
		Note that $F_0=0$. For the edges, for each $j=1,2,\ldots,q$, 
		$$
		F_{f(v_0)}+F_{f(v_{i+(j-1)\pi_p(0,b)})}=F_0+F_{i+(j-1)\pi_p(0,b)}\equiv F_i\mo{p}$$
		for $i=1,2,\ldots,\pi_p(0,b)-1$, and
		$$F_{f(v_0)}+F_{f(v_{\pi_p(0,b)+(j-1)\pi_p(0,b)})}=F_0+F_{\pi_p(0,b)+(j-1)\pi_p(0,b)}\equiv F_0\mo{p}.$$
		Calculating the Legendre symbol $(F_i/p)$ for $i=0,1,\ldots,\pi_p(0,b)-1$, we obtain
		$$e_{f_p^*}(0)=q[|\Lambda_{-1}^p(0,b)|+|\Lambda_{0}^p(0,b)|]\text{ and }e_{f_p^*}(1)=q|\Lambda_{1}^p(0,b)|.$$
		Clearly, $|e_{f_p^*}(0)-e_{f_p^*}(1)|=0$ because $p$ is a $0$-PL prime relative to $(0,b)$. Consequently, $S_{q\pi_p(0,b)+1}$ is a $(0,b)$-FLC graph modulo $p$.
	\end{proof}
	
	\begin{theorem}\normalfont\label{thm3}
		Let $0\in\chi(0,b)$. The wheel graph $W_{q\pi_p(0,b)+1}$ is a $(0,b)$-FLC graph modulo $p$, where $p$ is a $0$-PL prime relative to $(0,b)$, for all $q\geq 1$.
	\end{theorem}
	
	\begin{proof}
		Let $C_{q\pi_p(0,b)}$ be a cycle of $W_{q\pi_p(0,b)+1}$ with $V(C_{q\pi_p(0,b)})=\{v_1,v_2,\ldots,v_{q\pi_p(0,b)}\}$. Assume that $x_0$ is a vertex of $W_{q\pi_p(0,b)+1}$ for which $x_0$ is adjacent to all vertices of $C_{q\pi_p(0,b)}$ (see Figure~\ref{Wqpi}). Define the function $f:V(W_{q\pi_p(0,b)+1})\to \{0,1,\ldots,q\pi_p(0,b)\}$ by
		$$f(x_0)=0,\text{ and }$$
		$$f(v_i)=i\text{ for }i=1,2,\ldots,q\pi_p(0,b).$$
		Thus, $f$ is a bijective function.
		\begin{figure}[h]\centering
			\begin{tikzpicture}
				\node[circle,draw,scale=0.4,label=90:$x_0$] (v0) at (0,0) {};
				\node[circle,draw,scale=0.4,label=180:$v_1$] (v1) at (-1,-1) {};
				\node[circle,draw,scale=0.4,label=180:$v_2$] (v2) at (-1.5,0) {};
				\node[circle,draw,scale=0.4,label=180:$v_3$] (v3) at (-1,1) {};
				\node[rectangle] (cd) at (0,1) {$\ldots$};
				\node[circle,draw,scale=0.4,label=0:$v_{q\pi_p(0,b)-2}$] (vqpi-2) at (1,1) {};
				\node[circle,draw,scale=0.4,label=0:$v_{q\pi_p(0,b)-1}$] (vqpi-1) at (1.5,0) {};
				\node[circle,draw,scale=0.4,label=0:$v_{q\pi_p(0,b)}$] (vqpi) at (1,-1) {};
				
				\draw[-] (v0)--(v1);
				\draw[-] (v0)--(v2);
				\draw[-] (v0)--(v3);
				\draw[-] (v0)--(vqpi-2);
				\draw[-] (v0)--(vqpi-1);
				\draw[-] (v0)--(vqpi);
				
				\draw[-] (v0)--(v1)--(v2)--(v3)--(cd)--(vqpi-2)--(vqpi-1)--(vqpi)--(v1);
			\end{tikzpicture}
			\caption{Wheel graph $W_{q\pi_p(0,b)+1}$}\label{Wqpi}
		\end{figure}
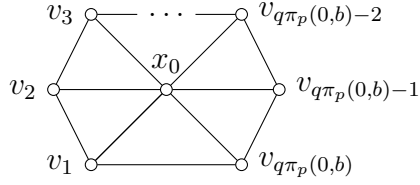

		For the edges of $C_{q\pi_p(0,b)}$, assume that 
		$$\mathbb{F}(i,j)=F_{f(v_{i+(j-1)\pi_p(0,b)})}+F_{f(v_{i+1+(j-1)\pi_p(0,b)})}.$$
		For each $j=1,2,\ldots,q$,
		\begin{align*}
			\mathbb{F}(i,j)&= F_{i+(j-1)\pi_p(0,b)}+F_{i+1+(j-1)\pi_p(0,b)}\\
			&\equiv F_{i}+F_{i+1}\mo{p}\\
			&\equiv F_{i+2}\mo{p}\\
			&\equiv\begin{cases}
				F_{i+2}\mo{p}&\text{ for }i=1,2,\ldots,\pi_p(0,b)-3\\
				F_0\mo{p}&\text{ for }i=\pi_p(0,b)-2\\
				F_1\mo{p}&\text{ for }i=\pi_p(0,b)-1\\
				F_2\mo{p}&\text{ for }i=\pi_p(0,b)\text{ and }j\neq q.
			\end{cases}
		\end{align*}
		Lastly, 
		\begin{align*}
			F_{f(v_{q\pi_p(0,b)})}+F_{f(v_{1})}&=F_{q\pi_p(0,b)}+F_1\\
			&\equiv F_0+F_1\mo{p}\\
			&\equiv F_2\mo{p}.
		\end{align*}
		Evaluating the Legendre symbol $(F_i/p)$ for $i=0,1,\ldots,\pi_p(0,b)-1$, it follows that the number of edges of $C_{q\pi_p(0,b)}$ with label $0$ and $1$ are $q[|\Lambda_{-1}^p(0,b)|+|\Lambda_{0}^p(0,b)|]$ and $q|\Lambda_1^p(0,b)|$, respectively, under $f_p^*$.
		
		Note that $F_0=0$. For the remaining edges, for each $j=1,2,\ldots,q$,
		$$
		F_{f(x_0)}+F_{f(v_{i+(j-1)\pi_p(0,b)})}=F_0+F_{i+(j-1)\pi_p(0,b)}\equiv F_i\mo{p}$$
		for $i=1,2,\ldots,\pi_p(0,b)-1$, and
		$$F_{f(x_0)}+F_{f(v_{\pi_p(0,b)+(j-1)\pi_p(0,b)})}=F_0+F_{\pi_p(0,b)+(j-1)\pi_p(0,b)}\equiv F_0\mo{p}.$$
		Calculating the Legendre symbol $(F_i/p)$ for $i=0,1,\ldots,\pi_p(0,b)-1$, the number of edges with label $0$ and $1$ are $q[|\Lambda_{-1}^p(0,b)|+|\Lambda_{0}^p(0,b)|]$ and $q|\Lambda_{1}^p(0,b)|$, respectively, under $f_p^*$.
		
		Consequently,
		$$e_{f_p^*}(0)=2q[|\Lambda_{-1}^p(0,b)|+|\Lambda_{0}^p(0,b)|]\text{ and }e_{f_p^*}(1)=2q|\Lambda_1^p(0,b)|.$$
		It follows that $|e_{f_p^*}(0)-e_{f_p^*}(1)|=0$ because $p$ is a $0$-PL prime relative to $(0,b)$. Therefore, $W_{q\pi_p(0,b)+1}$ is a $(0,b)$-FLC graph modulo $p$.
	\end{proof}
	\begin{theorem}\normalfont\label{thm4}
		Assume that $0\in\chi(a,b)$. Let $G$ be a connected graph and let $q$ be an integer with $q\geq 1$. The following graphs are $(a,b)$-FLC graphs modulo $p$, where $p$ is a $0$-PL prime relative to $(a,b)$.
		
		\begin{enumerate}
			\item[i.] Lexicographic product graph $C_{q\pi_p(a,b)}\otimes G$ with $p\equiv\pm 1\mo{8}$
			\item[ii.] Cartesian product graph $C_{q\pi_p(a,b)}\square G$ with $p\equiv\pm 1\mo{8}$
			\item[iii.] Tensor product graph $C_{q\pi_p(a,b)}\times G$
			\item[iv.] Strong product graph $C_{q\pi_p(a,b)}\boxtimes G$ with $p\equiv\pm 1\mo{8}$
		\end{enumerate}
	\end{theorem}
	
	\begin{proof}
		Assume that $G$ has order $n$ and size $m$. In addition, suppose that $V(C_{q\pi_p(a,b)})=\{u_0,u_1,\ldots,u_{q\pi_p(a,b)-1}\}$ (see Figure~\ref{Cqpi}) and $V(G)=\{v_1,v_2,\ldots,v_n\}$. Now, let $f:V(C_{q\pi_p(a,b)}\star G)\to \{0,1,\ldots,nq\pi_p(a,b)-1\}$ be a function defined by
		$$f((u_j,v_i))=j+(i-1)q\pi_p(a,b),$$
		where $\star\in \{\otimes, \square, \times, \boxtimes\}$, for $i=1,2,\ldots,n$ and $j=0,1,\ldots,q\pi_p(a,b)-1$. Clearly, $f$ is a bijective function.
		
		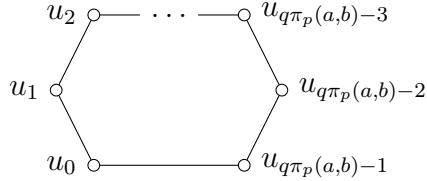
\begin{figure}[h]\centering
			\begin{tikzpicture}
				\node[circle,draw,scale=0.4,label=180:$u_0$] (v1) at (-1,-1) {};
				\node[circle,draw,scale=0.4,label=180:$u_1$] (v2) at (-1.5,0) {};
				\node[circle,draw,scale=0.4,label=180:$u_2$] (v3) at (-1,1) {};
				\node[rectangle] (cd) at (0,1) {$\ldots$};
				\node[circle,draw,scale=0.4,label=0:$u_{q\pi_p(a,b)-3}$] (vqpi-2) at (1,1) {};
				\node[circle,draw,scale=0.4,label=0:$u_{q\pi_p(a,b)-2}$] (vqpi-1) at (1.5,0) {};
				\node[circle,draw,scale=0.4,label=0:$u_{q\pi_p(a,b)-1}$] (vqpi) at (1,-1) {};

				\draw[-](v1)--(v2)--(v3)--(cd)--(vqpi-2)--(vqpi-1)--(vqpi)--(v1);
			\end{tikzpicture}
			\caption{Cycle graph $C_{q\pi_p(a,b)+1}$}\label{Cqpi}
		\end{figure}
		
		For (i), (ii), and (iv), consider the edges of the form $(u_l,v_r)(u_l,v_s)$ for $v_rv_s\in E(G)$ and $l=0,1,\ldots,q\pi_p(a,b)-1$. Let
		$$\mathbb{F}_1(j,t)=F_{f((u_{j+(t-1)\pi_p(a,b)},v_r))}+F_{f((u_{j+(t-1)\pi_p(a,b)},v_s))}.$$
		So, for each $t=1,2,\ldots,q$,
		\begin{align*}
			\mathbb{F}_1(j,t)&=F_{j+(t-1)\pi_p(a,b)+(r-1)q\pi_p(a,b)}+F_{j+(t-1)\pi_p(a,b)+(s-1)q\pi_p(a,b)}\\
			&\equiv F_j+F_j\mo{p}\\
			&\equiv 2F_j\mo{p}
		\end{align*}
		for $j=0,1,\ldots,\pi_p(a,b)-1$. Since $p\equiv \pm 1\mo{8}$, by Theorems~\ref{multiplicative} and \ref{2/p}, $(2F_j/p)=(F_j/p)$ for $j=0,1,\ldots,\pi_p(a,b)-1$. Because $G$ has size $m$, it is clear that the number of edges of $G$ with label $0$ and $1$ are $mq[|\Lambda_{-1}^p(a,b)|+|\Lambda_{0}^p(a,b)|]$ and $mq|\Lambda_1^p(a,b)|$, respectively, under $f_p^*$.
		
		For the edges of the form $(a,v_r)(b,v_s)$, where $ab\in E(C_{q\pi_p(a,b)})$ and the variables $r$ and $s$ depends on (i), (ii), (iii), and (iv), assume that
		$$\mathbb{F}_2(j,t)=F_{f((u_{j+(t-1)\pi_p(a,b)},v_r))}+F_{f((u_{j+1+(t-1)\pi_p(a,b)},v_s))}$$
		Thus, for each $t=1,2,\ldots,q$,
		\begin{align*}
			\mathbb{F}_2(j,t)&=F_{j+(t-1)\pi_p(a,b)+(r-1)q\pi_p(a,b)}+F_{j+1+(t-1)\pi_p(a,b)+(s-1)q\pi_p(a,b)}\\
			&\equiv F_{j}+F_{j+1}\mo{p}\\
			&\equiv F_{j+2}\mo{p}\\
			&\equiv\begin{cases}
				F_{j+2}\mo{p}&\text{ for }j=0,1,\ldots,\pi_p(a,b)-3\\
				F_0\mo{p}&\text{ for }j=\pi_p(a,b)-2\\
				F_1\mo{p}&\text{ for }j=\pi_p(a,b)-1\text{ and }t\neq q.
			\end{cases}
		\end{align*}
		Lastly,
		\begin{align*}
			F_{f((u_{q\pi_p(a,b)-1},v_r))}+F_{f((u_{0},v_s))}&=F_{q\pi_p(a,b)-1}+F_{0}\\
			&\equiv F_{q\pi_p(a,b)-1}+F_{\pi_p(a,b)}\mo{p}\\
			&\equiv F_{\pi_p(a,b)+1}\mo{p}\\
			&\equiv F_1\mo{p}.
		\end{align*}
		Hence, we calculate the Legendre symbol $(F_i/p)$ for $i=0,1,\ldots,\pi_p(a,b)-1$. Also, observe that $|\Lambda_{1}^p(a,b)|-|\Lambda_{0}^p(a,b)|-|\Lambda_{-1}^p(a,b)|=0$ since $p$ is a $0$-PL prime relative to $(a,b)$.
		
		For (i), we have $r,s=1,2,\ldots,n$. As a result,
		\begin{align*}
			e_{f_p^*}(0)&=mq[|\Lambda_{-1}^p(a,b)|+|\Lambda_{0}^p(a,b)|]+n^2q[|\Lambda_{-1}^p(a,b)|+|\Lambda_{0}^p(a,b)|]\\
			&=(mq+n^2q)[|\Lambda_{-1}^p(a,b)|+|\Lambda_{0}^p(a,b)|],\\
			e_{f_p^*}(1)&=m|\Lambda_{1}^p(a,b)|+n^2q|\Lambda_{1}^p(a,b)|\\
			&=(mq+n^2q)|\Lambda_{1}^p(a,b)|.
		\end{align*}
		It is clear that $|e_{f_p^*}(0)-e_{f_p^*}(1)|=0$ and so $C_{q\pi_p(a,b)}\oplus G$ is an $(a,b)$-FLC graph modulo $p$.
		
		For (ii), we have $r=s=t$ for $t=1,2,\ldots,n$. Thus,
		\begin{align*}
			e_{f_p^*}(0)&=mq[|\Lambda_{-1}^p(a,b)|+|\Lambda_{0}^p(a,b)|]+nq[|\Lambda_{-1}^p(a,b)|+|\Lambda_{0}^p(a,b)|]\\
			&=(mq+nq)[|\Lambda_{-1}^p(a,b)|+|\Lambda_{0}^p(a,b)|],\\
			e_{f_p^*}(1)&=mq|\Lambda_{1}^p(a,b)|+nq|\Lambda_{1}^p(a,b)|\\
			&=(mq+nq)|\Lambda_{1}^p(a,b)|.
		\end{align*}
		Clearly, $|e_{f_p^*}(0)-e_{f_p^*}(1)|=0$. Therefore, $C_{q\pi_p(a,b)}\square G$ is an $(a,b)$-FLC graph modulo $p$. 
		
		For (iii), $v_sv_r\in E(G)$. Because $$(u_j,v_r)(u_{j+1},v_s),(u_j,v_s)(u_{j+1},v_r)\in E(C_{q\pi_p}\times G)$$ and the size of $G$ is $m$, we have
		\begin{align*}
			e_{f_p^*}(0)&=2mq[|\Lambda_{-1}^p(a,b)|+|\Lambda_{0}^p(a,b)|],\\
			e_{f_p^*}(1)&=2mq|\Lambda_{1}^p(a,b)|.
		\end{align*}
		It is evident that $|e_{f_p^*}(0)-e_{f_p^*}(1)|=0$. Therefore, $C_{q\pi_p(a,b)}\times G$ is an $(a,b)$-FLC graph modulo $p$. 
		
		For (iv), note that $E(C_{q\pi_p(a,b)}\boxtimes G)=E(C_{q\pi_p(a,b)}\square G)\cup E(C_{q\pi_p(a,b)}\times G)$ and $E(C_{q\pi_p(a,b)}\square G)\cap E(C_{q\pi_p(a,b)}\times G)=\varnothing$. Thus, combining (ii) and (iii), it follows that
		\begin{align*}
			e_{f_p^*}(0)&=(mq+nq)[|\Lambda_{-1}^p(a,b)|+|\Lambda_{0}^p(a,b)|]+2mq[|\Lambda_{-1}^p(a,b)|+|\Lambda_{0}^p(a,b)|]\\
			&=(3mq+nq)[|\Lambda_{-1}^p(a,b)|+|\Lambda_{0}^p(a,b)|],\\
			e_{f_p^*}(1)&=(mq+nq)|\Lambda_{1}^p(a,b)|+2mq|\Lambda_{1}^p(a,b)|\\
			&=(3mq+nq)|\Lambda_{1}^p(a,b)|.
		\end{align*}
		Hence, $|e_{f_p^*}(0)-e_{f_p^*}(1)|=0$ and so $C_{q\pi_p(a,b)}\boxtimes G$ is an $(a,b)$-FLC graph modulo $p$.
	\end{proof}
	
	\begin{theorem}\normalfont\label{thm5}
		Let $k\geq 1$ be an integer with $k\in\chi(0,b)$. Suppose that $G$ is a connected graph of order $n$ and size $n(2k-1)+\varepsilon$, where $\varepsilon\in \{-1,0,1\}$. Moreover, let $(F_1/p)=(F_2/p)=1$. Then the corona graph $G\circ P_{\pi_p(0,b)-1}$ is a $(0,b)$-FLC graph modulo $p$, where $p$ is a $k$-PL prime relative to $(0,b)$. 
	\end{theorem}
	
	\begin{proof}
		Let $V(G)=\{v_0^1,v_0^2,\ldots,v_0^n\}$ and let  $P_{\pi_p(0,b)-1}^j$ be the $j$th copy of $P_{\pi_p(0,b)-1}$ with $V(P_{\pi_p(0,b)-1}^j)=\{v_1^j,v_2^j\ldots,v_{\pi_p(0,b)-1}^j\}$ for $j=1,2,\ldots,n$ (see Figure~\ref{Ppi}). Note that in the corona graph $G\circ P_{\pi_p(0,b)-1}$, $v_0^jv_i^j\in E(G\circ P_{\pi_p(0,b)-1})$ for $i=1,2,\ldots,\pi_p(0,b)-1$ and $j=1,2,\ldots,n$. Now, let $f:V(G\circ P_{\pi_p(0,b)-1})\to \{0,1,\ldots,n\pi_p(0,b)-1\}$ be a function defined by
		$$f(v_i^j)=i+(j-1)\pi_p(0,b)$$
		for $i=0,1,\ldots,\pi_p(0,b)-1$ and $j=1,2,\ldots,n$. Hence, $f$ is a bijective function.
		
		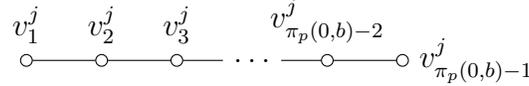
\begin{figure}[h]\centering
			\begin{tikzpicture}
				\node[circle,draw,scale=0.4,label=90:$v_1^j$] (v1) at (0,0) {};
				\node[circle,draw,scale=0.4,label=90:$v_2^j$] (v2) at (1,0) {};
				\node[circle,draw,scale=0.4,label=90:$v_3^j$] (v3) at (2,0) {};
				\node[rectangle] (cd) at (3,0) {$\ldots$};
				\node[circle,draw,scale=0.4,label=90:$v_{\pi_p(0,b)-2}^j$] (vqpi-1) at (4,0) {};
				\node[circle,draw,scale=0.4,label=0:$v_{\pi_p(0,b)-1}^j$] (vqpi) at (5,0) {};
				\draw[-] (v1)--(v2)--(v3)--(cd)--(vqpi-1)--(vqpi);
			\end{tikzpicture}
			\caption{Path graph $P_{\pi_p(0,b)-1}^j$}\label{Ppi}
		\end{figure}
		
		Note that $F_0=0$. For the edges of $G$, if $v_0^rv_0^s\in E(G)$, then 
		$$F_{f(v_0^r)}+F_{f(v_0^s)}=F_{(r-1)\pi_p(0,b)}+F_{(s-1)\pi_p(0,b)}\equiv 2F_0\equiv 0\mo{p}$$
		and it follows that all edges of $G$ are labeled by $0$ under $f_p^*$.
		
		For the edges of $P_{\pi_p(0,b)-1}^j$, let
		$$\mathbb{F}_1(i,j)=F_{f(v_i^j)}+F_{f(v_{i+1}^j)}.$$
		Thus, for each $j=1,2,\ldots,n$,
		\begin{align*}
			\mathbb{F}_1(i,j)&=F_{i+(j-1)\pi_p(0,b)}+F_{i+1+(j-1)\pi_p(0,b)}\\
			&\equiv F_i+F_{i+1}\mo{p}\\
			&\equiv F_{i+2}\mo{p}\\
			&\equiv\begin{cases}
				F_{i+2}\mo{p}&\text{ for }i=1,2,\ldots,\pi_p(0,b)-3\\
				F_0\mo{p}&\text{ for }i=\pi_p(0,b)-2.
			\end{cases}
		\end{align*}
		By calculating the Legendre symbol $(F_i/p)$ for $i=0,3,4,\ldots,\pi_p(0,b)-1$ and noting that $1,2\in \Lambda^p_{1}(0,b)$, then the number of edges of $P_{\pi_p(0,b)-1}^j$ with label $0$ and $1$, are $|\Lambda^p_{-1}(0,b)|+|\Lambda^p_{0}(0,b)|$ and $|\Lambda^p_{1}(0,b)|-2$, respectively, under $f_p^*$, for $j=1,2,\ldots,n$.
		
		For the edges of the form $v_0^jv_i^j$, suppose that
		$$\mathbb{F}_2(i,j)=F_{f(v_0^j)}+F_{f(v_i^j)}.$$
		Hence,
		\begin{align*}
			\mathbb{F}_2(i,j)&=F_{(j-1)\pi_p(0,b)}+F_{i+(j-1)\pi_p(0,b)}\\
			&\equiv F_i\mo{p}
		\end{align*}
		for $i=1,2,\ldots,\pi_p(0,b)-1$ and $j=1,2,\ldots,n$. Evaluating the Legendre symbol $(F_i/p)$ for $i=1,2,\ldots,\pi_p(0,b)-1$ and noting that $0\in \Lambda^p_0(0,b)$, then the number of edges of the form $v_0^jv_i^j$ with label $0$ and $1$ are $n[|\Lambda^p_{-1}(0,b)|+\Lambda^p_0(0,b)-1]$ and $n|\Lambda^p_1(0,b)|$, respectively, under $f_p^*$. 
		
		Consequently,
		\begin{align*}
			e_{f_p^*}(0)&=n(2k-1)+\varepsilon+n[|\Lambda^p_{-1}(0,b)|+|\Lambda^p_{0}(0,b)|]+n[|\Lambda^p_{-1}(0,b)|+|\Lambda^p_{0}(0,b)|-1]\\
			&=n(2k-1)+\varepsilon+n[2|\Lambda^p_{-1}(0,b)|+2|\Lambda^p_{0}(0,b)|-1],\text{ and }\\
			e_{f_p^*}(1)&=n[|\Lambda^p_{1}(0,b)|-2]+n|\Lambda^p_{1}(0,b)|\\
			&=n[2|\Lambda^p_{1}(0,b)|-2].
		\end{align*}
		Given that $p$ is a $k$-PL prime relative to $(0,b)$, that is, $k=|\Lambda^p_{1}(0,b)|-|\Lambda^p_{-1}(0,b)|-|\Lambda^p_{0}(0,b)|$, we have
		\begin{align*}
			e_{f_p^*}(0)-e_{f_p^*}(1)&=n(2k-1)+\varepsilon+n[2|\Lambda^p_{-1}(0,b)|+2|\Lambda^p_{0}(0,b)|-1-2|\Lambda^p_1(0,b)|+2]\\
			&=n(2k-1)+\varepsilon+n[-2(|\Lambda^p_{1}(0,b)|-|\Lambda^p_{-1}(0,b)|-|\Lambda^p_{0}(0,b)|)+1]\\
			&=n(2k-1)+\varepsilon+n(-2k+1)\\
			&=\varepsilon.
		\end{align*}
		Because $\varepsilon\in \{-1,0,1\}$, $|e_{f_p^*}(0)-e_{f_p^*}(1)|\leq 1$ and it follows that $G\circ P_{\pi_p(0,b)-1}$ is a $(0,b)$-FLC graph modulo $p$.
	\end{proof}

	\begin{theorem}\normalfont\label{thm6}
		Assume that $k\geq -1$ is an integer with $k\in\chi(0,b)$. Let $H=\bigcup_{j=1}^{\pi_p(0,b)-1}G^j$
		where $G^1,G^2,\ldots,G^{\pi_p(0,b)-1}$ are vertex-disjoint graphs of the same order $\pi_p(0,b)-1$ and the same size $m$. In addition, suppose that $G$ is a graph of order $\pi_p(0,b)-1$ and size $(k+1)(m+(\pi_p(0,b)-1)^2)+\varepsilon$, where $\varepsilon\in \{-1,0,1\}$. Then the join graph $G+H$ is a $(0,b)$-FLC graph modulo $p$, where $p$ is a $k$-PL prime relative to $(0,b)$ and $p\equiv \pm 1\mo{8}$. In addition, when $m=0$, the condition $p\equiv \pm 1\mo{8}$ is no longer required.
	\end{theorem}
	
	\begin{proof}
		Let $V(G)=\{v_0^1,v_0^2,\ldots,v_0^{\pi_p(0,b)-1}\}$. Suppose that $V(G^j)=\{v_1^j,v_2^j,\ldots,v_n^j\}$ for $j=1,2,\ldots,\pi_p(0,b)-1$. Let $f:V(G+H)\to \{0,1,\ldots,\pi_p(0,b)(\pi_p(0,b)-1)-1\}$ be a function defined by
		\begin{align*}
			f(v_i^j)&=j+(i-1)\pi_p(0,b)\text{ for }i=1,2,\ldots,\pi_p(0,b)-1,\\
			f(v_0^j)&=(j-1)\pi_p(0,b),
		\end{align*}
		for $j=1,\ldots,\pi_p(0,b)-1$. Evidently, $f$ is a bijective function. 
		
		Note that $F_0=0$. For the edges of $G$, for any edge $v_0^rv_0^s\in E(G)$, we have 
		$$F_{f(v_0^r)}+F_{f(v_0^s)}=F_{(r-1)\pi_p(0,b)}+F_{(s-1)\pi_p(0,b)}\equiv 2F_0\equiv 0\mo{p}.$$
		Hence, $f_p^*(e)=0$ for all $e\in E(G)$.
		
		For the edges of $G^j$, if $v_r^jv_s^j\in E(G^j)$, then
		$$F_{f(v_r^j)}+F_{f(v_s^j)}=F_{j+(r-1)\pi_p(0,b)}+F_{j+(s-1)\pi_p(0,b)}\equiv F_j+F_j\equiv 2F_j\mo{p}$$
		for $j=1,2,\ldots,\pi_p(0,b)-1$. Since $p\equiv \pm 1\mo{8}$, by Theorems~\ref{2/p} and \ref{multiplicative}, we have $(2F_j/p)=(F_j/p)$ for $j=1,2,\ldots,\pi_p(0,b)-1$. Because $0\in\Lambda^p_{0}(0,b)$, it follows that there are $m[|\Lambda^p_{-1}(0,b)|+|\Lambda^p_0(0,b)|-1]$ edges of $H$ with label $0$ and $m|\Lambda^p_1(0,b)|$ edges with label $1$, under $f_p^*$.
		
		For the edges of the form $v_0^lv_i^j$, we have
		$$F_{f(v_0^l)}+F_{f(v_i^j)}=F_{(l-1)\pi_p(0,b)}+F_{j+(i-1)\pi_p(0,b)}\equiv F_j\mo{p}$$
		for $i,l,j=1,2,\ldots,\pi_p(0,b)-1$. 
		Calculating the Legendre symbol $(F_j/p)$ for $j=1,2,\ldots,\pi_p(0,b)-1$ and noting that $0\in \Lambda^p_0$, it is evident that the number of edges of the form $v_0^lv_i^j$ with label $0$ and $1$ are $(\pi_p(0,b)-1)^2[|\Lambda^p_{-1}(0,b)|+|\Lambda^p_{0}(0,b)|-1]$ and $(\pi_p(0,b)-1)^2|\Lambda^p_1(0,b)|$, respectively, under $f_p^*$. 
		
		Therefore, 
		\begin{align*}
			e_{f_p^*}(0)&=(k+1)(m+(\pi_p(0,b)-1)^2)+\varepsilon+m[|\Lambda^p_{-1}(0,b)|+|\Lambda^p_0(0,b)|-1]\\
			&\hspace{0.5cm}+(\pi_p(0,b)-1)^2[|\Lambda^p_{-1}(0,b)|+|\Lambda^p_{0}(0,b)|-1]\\
			&=(k+1)(m+(\pi_p(0,b)-1)^2)+\varepsilon\\
			&\hspace{0.5cm}+(m+(\pi_p(0,b)-1)^2)[|\Lambda^p_{-1}(0,b)|+|\Lambda^p_{0}(0,b)|-1],\text{ and }\\
			e_{f_p^*}(1)&=m|\Lambda^p_1(0,b)|+(\pi_p(0,b)-1)^2|\Lambda^p_1(0,b)|\\
			&=(m+(\pi_p(0,b)-1)^2)|\Lambda^p_1(0,b)|.
		\end{align*}
		Note that $p$ is a $k$-PL prime relative to $(0,b)$, that is, $k=|\Lambda^p_1(0,b)|-|\Lambda^p_{-1}(0,b)|-|\Lambda^p_0(0,b)|$. Hence,
		\begin{align*}
			e_{f_p^*}(0)-e_{f_p^*}(1)&=(k+1)(m+(\pi_p(0,b)-1)^2)+\varepsilon\\
			&\hspace{0.5cm}+(m+(\pi_p(0,b)-1)^2)[|\Lambda^p_{-1}(0,b)|+|\Lambda^p_{0}(0,b)|-1-|\Lambda^p_{1}(0,b)|]\\
			&=(k+1)(m+(\pi_p(0,b)-1)^2)+\varepsilon\\
			&\hspace{0.5cm}+(m+(\pi_p(0,b)-1)^2)(-k-1)\\
			&=\varepsilon.
		\end{align*}
		This implies that $|e_{f_p^*}(0)-e_{f_p^*}(1)|\leq 1$ because $\varepsilon\in\{-1,0,1\}$. Consequently, $G+H$ is a $(0,b)$-FLC graph modulo $p$.
	\end{proof}

	\begin{theorem}\normalfont\label{thm7}
		Let $k\geq -1$ be an integer with $k\in\chi(0,b)$. Assume that $H$ is a graph of order $\pi_p(0,b)-1$ and size $m$. Moreover, suppose that $G$ is a connected graph of order $\pi_p(0,b)-1$ and size $(k+1)(m+\pi_p(0,b)-1)+\varepsilon$, where $\varepsilon\in \{-1,0,1\}$. Then the corona graph $G\circ H$ is an $(0,b)$-FLC graph modulo $p$, where $p$ is a $k$-PL prime relative to $(0,b)$ with $p\equiv \pm 1\mo{8}$. Furthermore, in the special case $m=0$, the condition $p\equiv \pm 1\mo{8}$ is unnecessary.
	\end{theorem}

	\begin{proof}
		Let $V(G)=\{v_0^1,v_0^2,\ldots,v_0^{\pi_p(0,b)-1}\}$. Assume that $H^j$ is the $i$th copy of $H$ with $V(H^j)=\{v_1^j,v_2^j,\ldots,v_{\pi_p(0,b)-1}^j\}$ for $j=1,2,\ldots,\pi_p(0,b)-1$. Note that $v_0^j$ is adjacent to $v_i^j$ for $i,j=1,2,\ldots,\pi_p(0,b)-1$. Now, suppose that $f:V(G\circ H)\to\{0,1,\ldots,\pi_p(0,b)(\pi_p(0,b)-1)-1\}$ is a function defined by
		\begin{align*}
			f(v_i^j)&=j+(i-1)\pi_p(0,b)\text{ for }i=1,2,\ldots,\pi_p(0,b)-1,\\
			f(v_0^j)&=(j-1)\pi_p(0,b),
		\end{align*}
		for $j=1,2,\ldots,\pi_p(0,b)-1$. Thus, $f$ is a bijective function.
		
		Note that $F_0=0$. For the edges of $G$, if $v_0^rv_0^s\in E(G)$, then
		$$F_{f(v_0^r)}+F_{f(v_0^s)}=F_{(r-1)\pi_p(0,b)}+F_{(s-1)\pi_p(0,b)}\equiv 2F_0\equiv 0\mo{p}$$
		which means $f_p^*(e)=0$ for each $e\in E(G)$.
		
		For the edges of $H^j$, if $v_r^jv_s^j\in E(H^j)$, then
		$$F_{f(v_r^j)}+F_{f(v_s^j)}=F_{j+(r-1)\pi_p(0,b)}+F_{j+(s-1)\pi_p(0,b)}\equiv F_j+F_j\equiv 2F_j\mo{p}.$$
		Note that $(2F_j/p)=(F_j/p)$, for $j=1,2,\ldots,\pi_p(0,b)-1$ by Theorems~\ref{2/p} and \ref{multiplicative}. Given that $0\in \Lambda^p_0(0,b)$, it is evident that the number of edges of $\bigcup_{j=1}^{\pi_p(0,b)-1}H^j$ with label $0$ and $1$ are $m[|\Lambda^p_{-1}(0,b)|+|\Lambda^p_0(0,b)|-1]$ and $m|\Lambda^p_1(0,b)|$, respectively, under $f_p^*$.
		
		For the edges of the form $v_0^jv_i^j$, we have
		$$F_{f(v_0^j)}+F(f(v_i^j))=F_{(j-1)\pi_p(0,b)}+F_{j+(i-1)\pi_p(0,b)}\equiv F_j\mo{p}$$
		for $i,j=1,2,\ldots,\pi_p(0,b)-1$. Now, by evaluating the Legendre symbol $(F_j/p)$ for $j=1,2,\ldots,\pi_p(0,b)-1$ and noting that $0\in \Lambda^p_{0}(0,b)$, it is clear that the number of edges of the form $v_0^jv_i^j$ with label $0$ and $1$ are $(\pi_p(0,b)-1)[|\Lambda^p_{-1}(0,b)|+|\Lambda^p_0(0,b)|-1]$ and $(\pi_p(0,b)-1)|\Lambda_1^p(0,b)|$, respectively, under $f_p^*$. 
		
		As a consequence,
		\begin{align*}
			e_{f_p^*}(0)&=(k+1)(m+\pi_p(0,b)-1)+\varepsilon+m[|\Lambda^p_{-1}(0,b)|+|\Lambda^p_0(0,b)|-1]\\
			&\hspace{0.5cm}+(\pi_p(0,b)-1)[|\Lambda^p_{-1}(0,b)|+|\Lambda^p_0(0,b)|-1]\\
			&=(k+1)(m+\pi_p(0,b)-1)+\varepsilon\\
			&\hspace{0.5cm}+(m+\pi_p(0,b)-1)[|\Lambda^p_{-1}(0,b)|+|\Lambda^p_0(0,b)|-1],\text{ and }\\
			e_{f_p^*}(1)&=m|\Lambda^p_1(0,b)|+(m+\pi_p(0,b)-1)|\Lambda^p_1(0,b)|\\
			&=(m+\pi_p(0,b)-1)|\Lambda^p_1(0,b)|.
		\end{align*}
		Since $p$ is a $k$-PL prime relative to $(0,b)$, that is, $k=|\Lambda^p_1(0,b)|-|\Lambda^p_{-1}(0,b)|-|\Lambda^p_0(0,b)|$, we have
		\begin{align*}
			e_{f_p^*}(0)-e_{f_p^*}(1)&=(k+1)(m+\pi_p(0,b)-1)+\varepsilon\\
			&\hspace{0.5cm}+(m+\pi_p(0,b)-1)[|\Lambda^p_{-1}(0,b)|+|\Lambda^p_0(0,b)|-1-|\Lambda^p_0(0,b)|]\\
			&=(k+1)(m+\pi_p(0,b)-1)+\varepsilon+(m+\pi_p(0,b)-1)(-k-1)\\
			&=\varepsilon.
		\end{align*}
		Therefore, $|e_{f_p^*}(0)-e_{f_p^*}(1)|\leq 1$ because $\varepsilon\in \{-1,0,1\}$. Consequently, $G\circ H$ is a $(0,b)$-FLC graph modulo $p$.
	\end{proof}
	
	\section{$k$-Pisano-Legendre Primes}\label{sec4}
	\hspace{1cm}In this section, all terms and notations of Definition~\ref{PL} will be used. In addition, let $p$ denotes an odd prime.
	
	\begin{proposition}\normalfont\label{pro1}
		If $p$ is a $k$-PL prime relative to $(a,b)$, then $|\Lambda_1^p(a,b)|=\frac{\pi_p(a,b)+k}{2}$.
	\end{proposition}
	
	\begin{proof}
		Note that
		\begin{align}
			|\Lambda_1^p(a,b)|+|\Lambda_{0}^p(a,b)|+|\Lambda_{-1}^p(a,b)|=\pi_p(a,b),\text{ and }\label{eq1}\\
			|\Lambda_1^p(a,b)|-|\Lambda_{-1}^p(a,b)|-|\Lambda_{0}^p(a,b)|=k.\label{eq2}
		\end{align}
		So, $2|\Lambda_1^p(a,b)|=\pi_p(a,b)+k$.
	\end{proof}
	
	\begin{corollary}\normalfont
		Suppose that $p$ is a $k$-PL prime relative to $(a,b)$. Then $k$ and $\pi_p(a,b)$ have the same parity. Thus, for all $k\in \chi (0,1)$, $k$ is even.
	\end{corollary}
	
	\begin{proof}
		By Proposition~\ref{pro1}, it is clear that $\pi_p(a,b)$ and $k$ have the same parity. In the classical Fibonacci sequence, it was proved in \cite{Trojovsky} that $\pi_m(0,1)$ is even for all $m\geq 3$. Hence, for any $k\in \chi (0,1)$, $k$ is even.
	\end{proof}
	
	\begin{corollary}\normalfont\label{cor1}
		An odd prime $p$ is a $k$-PL prime relative to $(a,b)$ if and only if $$|\Lambda_{-1}^p(a,b)|=\frac{\pi_p(a,b)-k-2|\Lambda_0^p(a,b)|}{2}\text{ or }|\Lambda_{0}^p(a,b)|=\frac{\pi_p(a,b)-k-2|\Lambda_{-1}^p(a,b)|}{2}.$$
	\end{corollary}

	\begin{proof}
		Immediately follows Proposition~\ref{pro1} and equations~(\ref{eq1}) and (\ref{eq2}).
	\end{proof}

	\begin{proposition}\normalfont\label{pro2}
		Let $z(p)$ define the order of apparition of $p$. Hence, $|\Lambda_0^p(0,1)|=\frac{\pi_p(0,1)}{z(p)}$. In fact, $|\Lambda_0^p(0,1)|\in \{1,2,4\}$. Moreover, $|\Lambda_0^p(0,1)|=4$ if and only if $z(p)$ is odd.
	\end{proposition}
	
	\begin{proof}
		In classical Fibonacci sequence, the order of apparition of an odd prime $p$, denoted by $z(p)$ (defined in \cite{Trojovsky}), is the smallest positive integer $n$ such that $F_n\equiv 0\mo{p}$. So, for any positive integer $m$, $F_{mz(p)}\equiv 0\mo{p}$. It is easy to show that $z(p)|\pi_p(0,1)$ because $F_{\pi_p(0,1)}\equiv 0\equiv F_{z(p)}\mo{p}$. This means that all $(0,1)$-Fibonacci numbers divisible by $p$ in one $(0,1)$-Pisano period occurs exactly at multiples of $z(p)$. Therefore, 
		$|\Lambda_0^p(0,1)|=\frac{\pi_p(0,1)}{z(p)}$.
		Using Lemma 3 in \cite{Trojovsky}, we have $|\Lambda_0^p(0,1)|\in \{1,2,4\}$ and it follows that $|\Lambda_0^p(0,1)|=4$ if and only if $z(p)$ is odd.
	\end{proof}
	
	Combining Corollary~\ref{cor1} and Proposition~\ref{pro2}, we have the following results:
	\begin{corollary}\normalfont\label{cor2}
		If $z(p)$ is odd, then $p$ is a $k$-PL prime relative to $(0,1)$ if and only if 
		$$|\Lambda_{-1}^p(0,1)|=\frac{\pi_p(0,1)-k}{2}-4.$$
	\end{corollary}
	
	\begin{corollary}\normalfont\label{cor3}
		An odd prime $p$ is a $k$-PL prime relative to $(0,1)$ if and only if 
		$$|\Lambda_{-1}^p(0,1)|=\frac{\pi_p(0,1)-k-2\epsilon}{2}$$
		for some $\epsilon\in \{1,2,4\}$.
	\end{corollary}
	
	\begin{table}[h]\centering
		\begin{tabular}{|c|c||c|c||c|c||c|c|}
			\hline
			$k$&$\zeta_{(0,1)}(k)$&$k$&$\zeta_{(0,1)}(k)$&$k$&$\zeta_{(0,1)}(k)$&$k$&$\zeta_{(0,1)}(k)$\\\hline
			$-42$ &$13591$&$-20$ &$97$&$0$  &$11$  &$20$ &$89$   \\\hline
			$-40$ &$5591$&$-18$ &$1289$&$2$  &$809$ &$22$ &$1231$ \\\hline
			$-38$ &$5689$&$-16$ &$1151$&$4$  &$113$ &$24$ &$1409$ \\\hline
			$-36$ &$1153$&$-14$ &$1009$&$6$  &$199$ &$28$ &$73$   \\\hline
			$-34$ &$14821$&$-12$ &$619$&$8$  &$41$  &$32$ &$1871$ \\\hline
			$-32$ &$241$&$-10$ &$461$&$10$ &$331$ &$34$ &$5741$ \\\hline
			$-28$ &$569$&$-8$ &$709$&$12$ &$17$  &$36$ &$3499$ \\\hline
			$-26$ &$1471$&$-6$ &$2251$&$14$ &$541$ &$38$ &$3391$ \\\hline
			$-24$ &$1031$&$-4$ &$5$&$16$ &$1999$&$40$ &$3919$ \\\hline
			$-22$ &$10771$&$-2$ &$7$&$18$ &$811$ &$42$ &$14969$\\\hline
		\end{tabular}
		\caption{Values of $\zeta_{(0,1)}(k)$ for some $k$}\label{tab2}
	\end{table}

	Moreover, an interesting pattern emerge from the function $\zeta_{(a,b)}:\chi(a,b)\to \mathbb{P}$, where $\mathbb{P}$ is the set of odd primes, defined by $$\zeta_{(a,b)}(k)=\min\{p:\text{$p$ is a $k$-PL prime relative to $(a,b)$}\},$$
	for integers $a$ and $b$. Table~\ref{tab2} exhibits some values of $\zeta_{(0,1)}(k)$ for some integer $k$.
	
	For the classical Fibonacci sequence case, Figure~\ref{Fig1} shows the plot of
	$\zeta_{(0,1)}(k)$ for some values of $k$. The behaviour of
	$\zeta_{(0,1)}(k)$ appears to be bounded below by some quadratic functions of
	the form $g(k)=ck^2$, where $c$ is a real number. In fact, Figure~\ref{Fig2}
	shows the reference curve $g(k)=0.055k^2$, which exhibits that
	$\zeta_{(0,1)}(k)>0.055k^2$ for $|k|<600$ and
	$3<\zeta_{(0,1)}(k)<25000$.
	
	This behaviour is not limited to $(a,b)=(0,1)$. Figure~\ref{Fig3}, for the
	classical Lucas sequence case, shows that the values of
	$\zeta_{(2,1)}(k)$ are above the given reference curve
	$g(k)=0.061k^2$. Furthermore, similar behaviour is observed for the cases
	$(a,b)=(-2,7)$ and $(a,b)=(-2,-4)$, as shown in Figures~\ref{Fig4}
	(with reference curve $g(k)=0.042k^2$) and \ref{Fig5}
	(with reference curve $g(k)=0.055k^2$).
	
	These numerical evidences lead to the following
	conjecture:
	
	\begin{conjecture}\normalfont
		For any integers $a$ and $b$ with $(a,b)\neq(0,0)$, there exists a real number
		$c$ with $0<c<1$ such that $\zeta_{(a,b)}(k)>ck^2$ for all
		$k\in \chi(a,b)$.
	\end{conjecture}
	
	In addition, Figures~\ref{Fig2} to \ref{Fig5} display many values of $k$ for
	each of the initial values $(a,b)=(0,1),(2,1),(-2,7),(-2,-4)$. This serves as
	evidence that there might be infinitely many values of $k$ with
	$k\in \chi(a,b)$ for fixed initial values $(a,b)$. Thus, this leads to the
	following general conjecture:
	
	\begin{conjecture}\normalfont\label{con1}
		For any integers $a$ and $b$ with $(a,b)\neq(0,0)$,
		$|\chi(a,b)|=\aleph_0$.
	\end{conjecture}
	
	Note that $\aleph_0$ is the symbol reserved for the cardinality of a countably infinite set.
	
	\begin{figure}[!hbt]\centering
		\includegraphics[height=3in,width=5in]{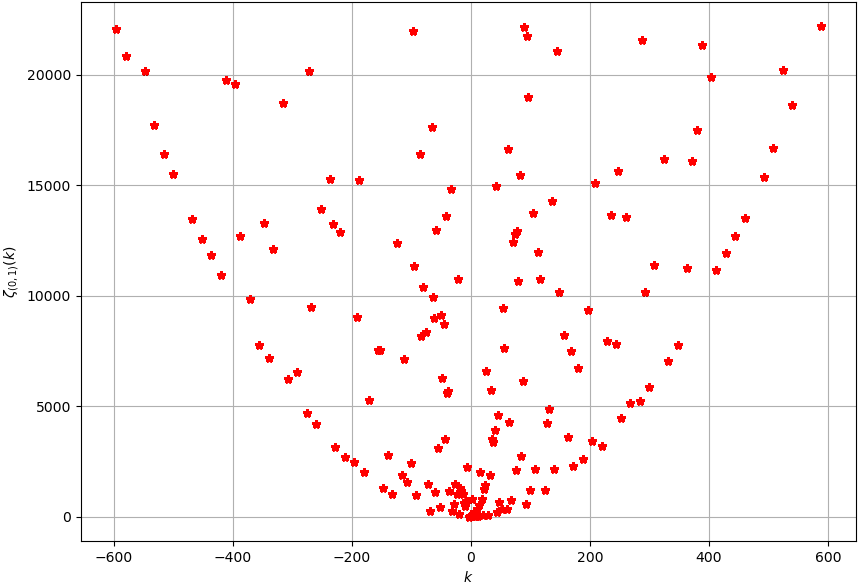}
		\caption{Plot of $\zeta_{(0,1)}(k)$}\label{Fig1}
	\end{figure}

	\begin{figure}[!hbt]\centering
		\includegraphics[height=3in,width=5in]{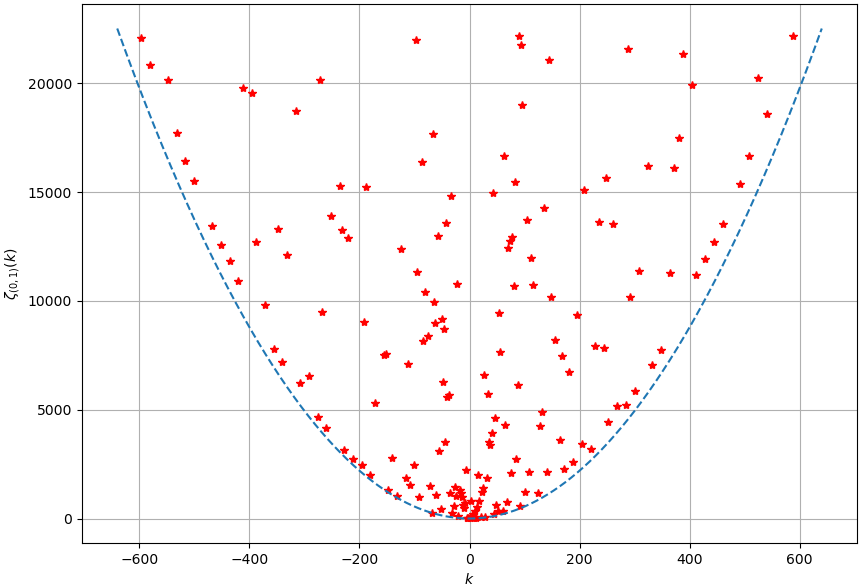}
		\caption{Plot of $\zeta_{(0,1)}(k)$}\label{Fig2}
	\end{figure}
	
\newpage
	
	\begin{figure}[!hbt]\centering
		\includegraphics[height=3in,width=5in]{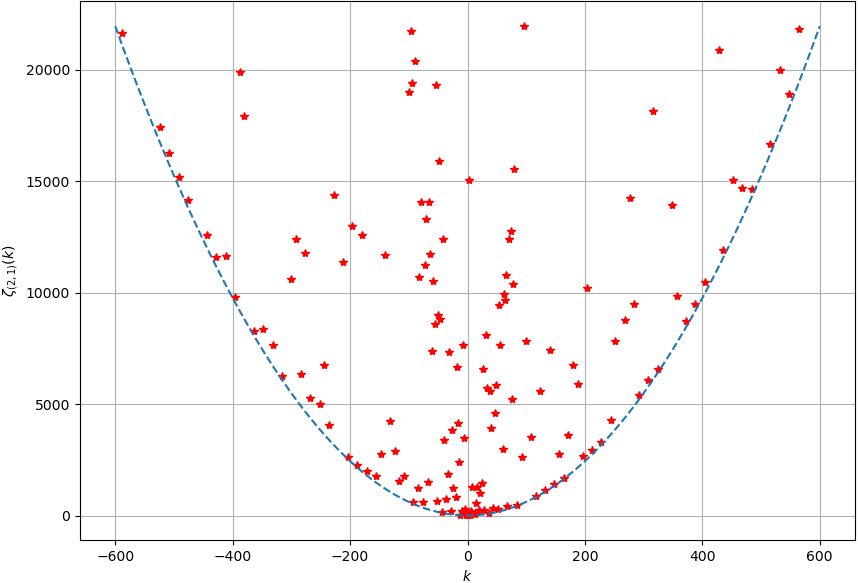}
		\caption{Plot of $\zeta_{(2,1)}(k)$}\label{Fig3}
	\end{figure}

	\begin{figure}[!hbt]\centering
		\includegraphics[height=3in,width=5in]{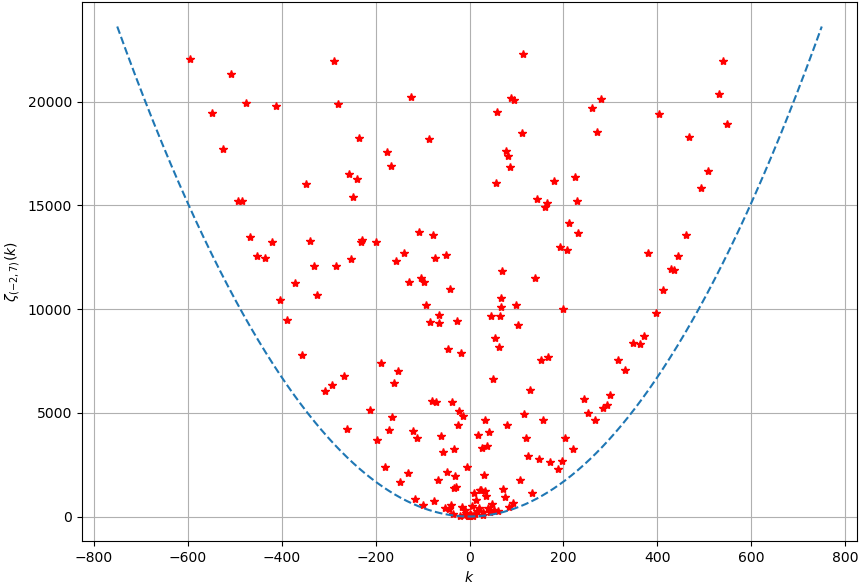}
		\caption{Plot of $\zeta_{(-2,7)}(k)$}\label{Fig4}
	\end{figure}
	
	\vspace{1cm}
	
	\begin{figure}[!hbt]\centering
		\includegraphics[height=3in,width=5in]{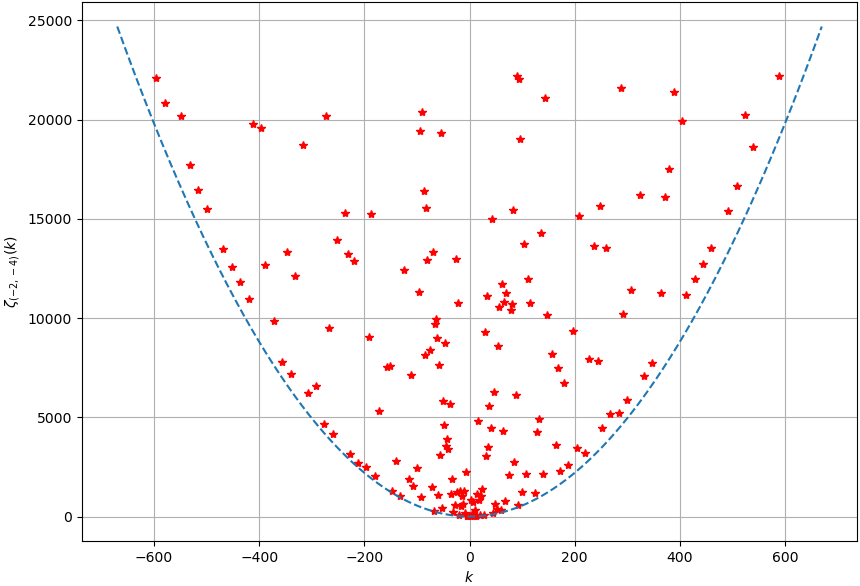}
		\caption{Plot of $\zeta_{(-2,-4)}(k)$}\label{Fig5}
	\end{figure}
	\newpage

	Most of the results in Section~\ref{sec3} have assumptions that $0\in \chi(a,b)$ but did not explain the existence of $0$-PL prime relative to $(a,b)$. Assume that $$\vartheta_{(a,b)}(k)=\{p:\text{$p$ is a $k$-PL prime relative $(a,b)$}\}$$ 
	where $k\in\chi(a,b)$. Table~\ref{tab3} displays the first ten elements of $\vartheta_{(a,b)}(0)$ for some integers $a$ and $b$.
	\begin{table}[h]\centering
		\begin{tabular}{|c|c|}
			\hline
			$(a,b)$& First ten elements of $\vartheta_{(a,b)}(0)$\\\hline
			$(0,1)$&$11, 29, 31, 71, 131, 191, 229, 251, 271, 281$\\\hline
			$(2,1)$&$5, 13, 37, 53, 61, 109, 149, 151, 157, 173$\\\hline
			$(5,12)$&$11, 13, 19, 31, 37, 47, 61, 101, 107, 109$\\\hline
			$(-9,-21)$&$13, 19, 29, 37, 47, 53, 61, 71, 79, 107$\\\hline
			$(-23,-20)$&$31, 47, 59, 61, 71, 107, 109, 149, 173, 191$\\\hline
			$(-232,-200)$&$13, 19, 47, 59, 61, 71, 101, 109, 173, 179$\\\hline
			$(80,21)$&$13, 19, 29, 31, 37, 47, 79, 149, 157, 173$\\\hline
			$(-3,-7)$&$13, 19, 29, 31, 37, 47, 53, 61, 71, 101, 107$\\\hline
			$(17,102)$&$19, 31, 37, 47, 59, 61, 71, 79, 101, 107$\\\hline
			$(6,7)$&$11, 19, 37, 47, 59, 61, 101, 107, 109, 157$\\\hline
		\end{tabular}
		\caption{First ten elements of $\vartheta_{(a,b)}(0)$ for some choice of $a$ and $b$}\label{tab3}
	\end{table}
	
	Suppose that 
	$$
	\varpi_{(a,b)}^n(k)=\{p:\text{ $p$ is a $k$-PL prime relative to $(a,b)$ with $p\leq n$}\}
	$$
	and let 
	$$
	\omega_{(a,b)}^n(k)=|\varpi_{(a,b)}^n(k)|,
	$$
	where $k\in\chi(a,b)$. Obviously,
	$$
	\lim_{n\to \infty}\omega_{(a,b)}^n(k)=|\vartheta_{(a,b)}(k)|
	$$
	for every $k\in\chi(a,b)$. Take $k=0$. Figure~\ref{Fig6} illustrates the behavior of
	$\omega_{(a,b)}^n(0)$ for the initial values $(a,b)=(0,1),(2,1),(7,6),(-5,-2),(-8,-2)$,
	showing a consistent increase in the number of $0$-PL primes as $n$ grows.
	In the classical Fibonacci sequence case $(0,1)$, $\omega_{(0,1)}^n(0)$ increases
	steadily, indicating that $0$-PL primes occur frequently among odd primes.
	A similar unbounded growth is observed for the classical Lucas sequence case $(2,1)$.
	Moreover, for the initial values $(a,b)=(7,6),(-5,-2),(-8,-2)$, the growth of
	$\omega_{(a,b)}^n(0)$ appears to be faster compared to the classical cases $(0,1)$ and
	$(2,1)$. These numerical observations collectively motivate the formulation of
	Conjecture~\ref{con3}, presented as follows:
	
	\begin{conjecture}\normalfont\label{con3}
		Suppose that $a$ and $b$ are integers with $(a,b)\neq (0,0)$. If $0\in\chi(a,b)$, then $|\vartheta_{(a,b)}(0)|=\aleph_0$. 
	\end{conjecture}
	
	Corollary~\ref{cor3} provides an equivalent formulation of Conjecture~\ref{con3} for the classical Fibonacci sequence case:
	\begin{corollary}\normalfont
		The cardinality $|\vartheta_{(0,1)}(0)|=\aleph_0$ if and only if there are infinitely many odd primes $p$ with $|\Lambda_{-1}^p(0,1)|=\frac{\pi_p(0,1)-2\epsilon}{2}$ for some $\epsilon\in\{1,2,4\}$. 
	\end{corollary}
	
	\begin{figure}[!hbt]\centering
		\includegraphics[height=3.5in,width=5in]{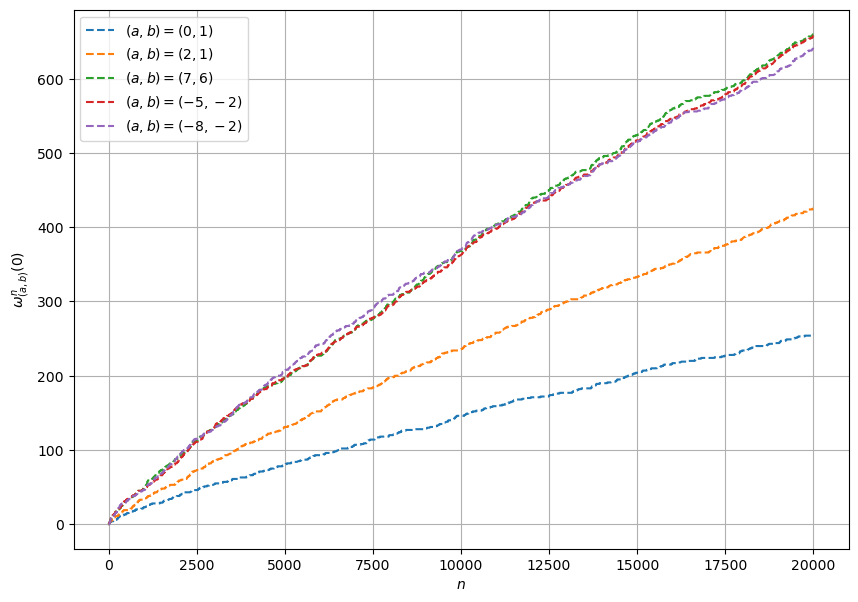}
		\caption{Plot of $\omega_{(a,b)}^n(0)$ for different values of $(a,b)$ with $0\leq n\leq 20000$}\label{Fig6}
	\end{figure}

	\section{Open Problems and Future Work}
	
	\hspace{1cm}In this paper, we have investigated the $(a,b)$-Fibonacci-Legendre cordial labeling of
	path graphs, star graphs, and wheel graphs, as well as graphs obtained through the graph
	operations union, join, corona, lexicographic product, cartesian product, tensor product,
	and strong product. These results demonstrate how the arithmetic structure induced by the
	$(a,b)$-Pisano period and the distribution of Legendre symbols of $(a,b)$-Fibonacci
	numbers modulo an odd prime $p$ can be exploited to obtain balanced edge labelings.
	
	The graph families considered in this study represent only a limited portion of known graphs. It is therefore natural to extend the study of $(a,b)$-Fibonacci-Legendre cordial
	labeling to other graphs not covered in this work, such as complete graphs, complete
	multipartite graphs, fan graphs, and other families of graphs. Studying these graphs may
	reveal new structural constraints or arithmetic phenomena arising from higher graph
	density and symmetry.
	
	The results in Section~\ref{sec3} rely fundamentally on the existence and arithmetic behavior of
	$k$-PL primes relative to $(a,b)$. This leads to the following open problems, which are directly motivated by the conjectures and numerical observations
	presented in Section~\ref{sec4}.
	
	\begin{question}\normalfont
		Is it possible for the constant number $c$ in Conjecture~\ref{con1} to have a lower bound of $0.01$ and an upper bound of $0.1$?
	\end{question}
	
	\begin{question}\normalfont
		Suppose that $a$ and $b$ are integers with $(a,b)\neq (0,0)$. For every $k\in\chi(a,b)$, is $|\vartheta_{(a,b)}(k)|=\aleph_0$?
	\end{question}
	
	Let $a_1$, $a_2$, $b_1$, and $b_2$ be integers with $(a_1,b_1)\neq (a_2,b_2)$, $(a_1,b_1)\neq (0,0)$, and $(a_2,b_2)\neq (0,0)$.
	
	\begin{question}\normalfont
		Are there values of $a_1$, $a_2$, $b_1$, and $b_2$ for which $\chi(a_1,b_1)=\chi(a_2,b_2)$?
	\end{question}
	
	\begin{question}\normalfont
		Are there integers $k_1\in\chi(a_1,b_1)$ and $k_2\in \chi(a_2,b_2)$ such that $\vartheta_{(a_1,b_1)}(k_1)=\vartheta_{(a_2,b_2)}(k_2)$?
	\end{question}
	
	\begin{question}\normalfont
		For a fixed positive integer $n$, are there integers $k_1\in\chi(a_1,b_1)$ and $k_2\in \chi(a_2,b_2)$ such that $\varpi_{(a_1,b_1)}^n(k_1)=\varpi_{(a_2,b_2)}^n(k_2)$?
	\end{question}
	
	These open problems and conjectures (presented in Section~\ref{sec4}) suggest several promising directions for future
	research. Progress on these questions may be achieved through a deeper analysis of the
	arithmetic properties of $(a,b)$-Pisano periods and the distribution of the Legendre symbols
	of $(a,b)$-Fibonacci numbers modulo primes. Further computational evidence, combined with
	analytic and probabilistic methods in number theory, may lead to partial or complete
	resolutions of the questions and conjectures.

\end{document}